\documentclass[a4paper,11pt]{article}
\usepackage{amsfonts}
\usepackage{pifont}
\usepackage[title]{appendix}
\usepackage{caption}
\usepackage{graphicx, subfig}
\usepackage{bm}
\usepackage{latexsym,amsmath,amssymb,cite,amsthm}
\usepackage{color,eucal,enumerate,mathrsfs}
\usepackage[normalem]{ulem}
\usepackage{amsmath}
\usepackage[pagewise]{lineno}
\usepackage[pagebackref=true, colorlinks, linkcolor=blue,  anchorcolor=blue,citecolor=blue]{hyperref}
\usepackage{amsthm}

\numberwithin{equation}{section}
\theoremstyle{plain}
\newtheorem{theorem}{Theorem}[section]
\newtheorem{corollary}[theorem]{Corollary}
\newtheorem{lemma}[theorem]{Lemma}
\newtheorem{example}[theorem]{Example}
\newtheorem{proposition}[theorem]{Proposition}
\theoremstyle{definition}
\newtheorem{definition}[theorem]{Definition}

\theoremstyle{remark}

\newtheorem{remark}[theorem]{Remark}

\newcommand\bdf{\begin{definition}}
	\newcommand\bpr{\begin{proposition}}
		\newcommand\brk{\begin{remark}}
			\newcommand\blm{\begin{lemma}}
				\newcommand\bexe{\begin{exercise}}
					\newcommand\bexa{\begin{example}}
						\newcommand\beqn{\begin{eqnarray*}}
							\newcommand\edf{\end{definition}}
						\newcommand\epr{\end{proposition}}
					\newcommand\erk{\end{remark}}
				\newcommand\elm{\end{lemma}}
			\newcommand\eexe{\end{exercise}}
		\newcommand\eexa{\end{example}}
	\newcommand\eeqn{\end{eqnarray*}}

\newcommand{\mm}{\mathfrak m}
\newcommand{\ms}{(X,\d,\mm)}

\newcommand{\rcdkn}{{\rm RCD}(K, N)}
\newcommand{\rcd}{{\rm RCD}(K, \infty)}

\newcommand{\E}{\mathbb{E}}

\newcommand{\N}{\mathbb{N}}

\newcommand{\R}{\mathbb{R}}
\newcommand{\K}{\mathrm{K}}
\newcommand{\rS}{\mathrm{S}}

\newcommand{\diam}{\mathop{\rm diam}\nolimits} 

\newcommand{\Lip}{\mathop{\rm Lip}\nolimits}

\renewcommand{\d}{{\mathrm d}}

\newcommand{\D}{{\mathrm D}}
\newcommand{\restr}[1]{\lower3pt\hbox{$|_{#1}$}}

\newcommand{\nchi}{{\raise.3ex\hbox{$\chi$}}}
\newcommand{\weakto}{\rightharpoonup}
\newcommand{\limi}{\varliminf}
\newcommand{\lims}{\varlimsup}

\title{{\bf Stability of optimal transport on metric measure spaces}
}

\begin{document}
	\author{Bang-Xian Han\thanks{School of   Mathematics,  Shandong University,  Jinan, China.  Email: hanbx@sdu.edu.cn. }
		\and Zhuo-Nan Zhu
		\thanks{School of Mathematical Sciences, University of Science and Technology of China, Hefei, China.  Email: zhuonanzhu@mail.ustc.edu.cn}
	}

	\date{\today} 
	\maketitle
	
	\begin{abstract}
		We prove  a quantitative  stability of Kantorovich potentials on non-smooth metric measure spaces with synthetic lower Ricci curvature bound,  thereby confirming a recent conjecture of  Kitagawa,  Letrouit and M\'erigot.   Our proof,  which employs the heat kernel-regularized $c$-transform,   does not rely on linear structure or  sectional curvature bounds,   is new even in the smooth setting.  
		As a corollary,  we get a quantitative stability of optimal transport maps on Alexandrov spaces with  lower curvature bound.  
		
	\end{abstract}

	
	\setcounter{tocdepth}{1}
	\tableofcontents

	\section{Introduction}
	\subsection{Background and Motivation}
	Optimal transport, initiated by Monge  \cite{monge1781} and reframed by  Kantorovich \cite{zbMATH03099866},  seeks the most efficient way to redistribute mass between two probability distributions.   Precisely, given two probability measures $\rho$ and $\mu$ defined on Polish spaces $X$ and $Y$ respectively, and a  cost function $c:X\times Y\rightarrow \R$, the Monge optimal transport problem aims to find a minimizer  $T$,   called an optimal transport map,   of the following optimization problem among all measurable maps $H: X \to Y$ pushing $\rho$ forward to $\mu$:
	\begin{equation}\label{2.1}
		\int_X c(x, T(x))  \,\d\rho(x)=\inf_{H_\sharp \rho = \mu}\int_X c(x, H(x))  \,\d\rho(x). \tag{MP}
	\end{equation}
 The   Monge problem is not always well-posed,   since it  prevents mass splitting.    Its relaxation,  called Kantorovich problem, optimizes over joint couplings instead of deterministic maps. By Kantorovich duality,   the Kantorovich problem is equivalent to the  following dual problem:
	\begin{equation}
		\sup_{ \phi(x) + \psi(y) \leq c(x, y)} \left\{ \int_X \phi(x) \, \d\rho(x) + \int_Y \psi(y) \, \d\mu(y)  \right\}.\tag{KD}
	\end{equation}
	The optimal functions $\phi,\psi$,  which always exist under standard assumptions,   are called  Kantorovich potentials \cite{AG-U, V-O}.  
	
	In the Euclidean space,   for $c(x,y)=\frac{1}{2}|x-y|^2$, Brenier's landmark result \cite{brenier1991polar} established that, for absolutely continuous source measures, the optimal transport map is unique and takes the form $T(x)=\nabla u(x)$ for a convex function $u$. This was later extended to Riemannian manifolds by McCann \cite{mccann2001polar}, who showed that for $c(x,y)=\frac{1}{2}\d^2(x,y)$, the optimal transport map is given by $T(x)=\exp_x(-\nabla\varphi(x))$ for a  Kantorovich potential $\varphi$.
	
	 A fundamental question in both theoretical and applied contexts of optimal transport is the quantitative stability of optimal transport maps and Kantorovich potentials under perturbations of the target measure (see Letrouit's lecture note \cite{LetrouitLecture} and the references therein).   Based on recent breakthroughs in the quantitative stability of optimal transport  on  Euclidean spaces \cite{zbMATH07794624, arXiv:2411.04908},    on boundaries of convex bodies \cite{Kitagawa12,  Kitagawa25}  and  on Riemannian manifolds  \cite{arXiv:2504.05412},     
Kitagawa--Letrouit--M\'erigot   \cite[\S 1.2]{arXiv:2504.05412} conjecture that:
	\begin{center}
		\emph{the quantitative stability results are also true in more general metric measure spaces with synthetic curvature bounds.}
	\end{center}

	\subsection{General Setting}
	We confirm the conjecture of Kitagawa--Letrouit--M\'erigot  in the  following setting.
	
\paragraph{A.  Metric measure spaces:}	An Alexandrov space is a  geodesic  space of finite Hausdorff dimension and of curvature bounded from below (cf.  \cite{zbMATH01626771}).   An ${\rm RCD}(K,N)$ space  is  a metric measure spaces verifying the synthetic Riemannian curvature-dimension condition  \cite{zbMATH05049052,zbMATH05578758, zbMATH06303881, G-O}.   An $n$-dimensional Alexandrov space with curvature bounded from below by $k$,  equipped with its $n$-Hausdorff measure,   is an ${\rm RCD}(k(n-1), n)$ space  \cite{zbMATH06032507,zbMATH05962558}.
		 After \cite{brenier1991polar, mccann2001polar},   Gigli--Rajala--Sturm \cite[Theorem 1.1]{ gigli2016optimal}  proved the existence and uniqueness of the optimal transport map on an $\rm{RCD}(K,N)$ space $(X,\d,\mm)$,  for the quadratic cost $c(x,y)=\frac{1}{2}\d^2(x,y)$ and the source measure $\rho\ll \mm$.

\paragraph{B.  Source measures:}	In  \cite[Theorem 1.9]{arXiv:2411.04908}, it has been shown that when the source measure $\rho$ is the uniform density on some non-John domain  $S\subset \R^n$, then no quantitative stability estimates of the form $$\|\phi_\mu-\phi_\nu\|_{L^2(\rho)}\leq C W_p^\alpha(\mu,\nu).$$    A bounded open subset $S$ of a metric space is called a \emph{John domain} if there is a distinguished point $x_0 \in S$ and a constant $\eta > 0$ such that, for every $x \in S$, there is a rectifiable curve $\gamma \colon [0, \ell(\gamma)] \to S$ parametrized by arc length, such that $\gamma(0) = x$, $\gamma(\ell(\gamma)) = x_0$, and
		\begin{equation*}
			\d(\gamma(t), S^c) \geq \eta t,~~~\forall t \in [0, \ell(\gamma)].
		\end{equation*}
	John domains  encompass many cases of interest, such as  bounded Lipschitz domains, bounded domains satisfying a cone condition and certain fractal domains (see \cite{zbMATH00858510} for more discussions).

	Moreover, examples found by Letrouit \cite{letrouit2025unstableoptimaltransportmaps} indicates that both unboundedness of the density of the source measure $\rho$,  and  the openness of    $S$ may cause  instability  of optimal transport maps.  So we assume that $a_1\mm\restr{S} \leq \rho\leq a_2\mm\restr{S}$ for  some John domain $S$ and   constants $a_1,a_2>0$.

	\paragraph{C.  Kantorovich potentials:}  It is also necessary to  (zero-mean) normalize the Kantorovich potential  $\phi$ from $\rho$  such that $\E_\rho(\phi)=\int_S\phi\,\d\rho=0$.  Together with the uniqueness of the optimal transport map, such  $\phi$ is  unique and thus makes sense to talk about its stability.

	\subsection{Main Results}
	Our main theorem concerns the quantitative  $L^1$ stability of Kantorovich potentials on ${\rm RCD}(K,N)$  spaces.  
	
	\begin{theorem}\label{thm0.1}
		Let $(X,\d,\mm)$ be an ${\rm RCD}(K,N)$ metric measure space. Let $S\subseteq X$ be a John domain and $Y\subseteq X$ be compact with $\mm(Y)>0$. Let $\rho\in \mathcal{P}(S)$  be with $a_1\mm\restr{S} \leq \rho\leq a_2\mm\restr{S}$ for some positive constants $a_1,a_2$. Then there exists a constant $C>0$, depending on $K,N,a_1,a_2, S, \diam(S\cup Y)$, such that for any $\mu,\nu\in\mathcal{P}(Y)$, 
		\begin{equation}
			\|\phi_\mu-\phi_\nu\|_{L^1(\rho)}\leq C W_1^\frac{1}{2}(\mu,\nu),
		\end{equation}
		where $\phi_\mu$ and $\phi_\nu$ are the Kantorovich potentials from $\rho$ to $\mu$ and $\rho$ to $\nu$ respectively.
	\end{theorem}
	
	In particular,  if $(S,  \d,  \mm)$  is a compact $\rcdkn$ space,    $S$ is surely a John domain.  So we have the following corollary.
	
	\begin{corollary}\label{cor}
		Let $(X,\d,\mm)$ be a compact ${\rm RCD}(K,N)$ metric measure space.  Let $\rho\in \mathcal{P}(X)$ be with $a_1\mm \leq \rho\leq a_2\mm$ for some positive constants $a_1,a_2$. Then the conclusion of Theorem \ref{thm0.1} holds.
	\end{corollary}
		
	Adapting the  strategy of \cite{arXiv:2504.05412}, we can also prove the stability of optimal transport maps on Alexandrov spaces.
		\begin{theorem}\label{thm0.2}
		Let $(X,  \d)$ be an $n$-dimensional Alexandrov space with curvature bounded from below by $k$, $\mm$ be the $n$-Hausdorff measure. Under the same assumption for $S, Y$ and $\rho$ as in Theorem \ref{thm0.1}, and if $S$ additionally has finite perimeter, then there exists a constant $C>0$, depending on $k,n,a_1, a_2, \diam(S\cup Y),   S$, such that for any $\mu,\nu\in\mathcal{P}(Y)$, 
		\begin{equation}
			\int_S \d^2(T_\mu(x), T_\nu(x))\,\d \rho(x) \leq CW_1^{1/6}(\mu, \nu),
		\end{equation}
		where $T_\mu$ and $T_\nu$ are the optimal transport maps from $\rho$ to $\mu$ and $\rho$ to $\nu$ respectively.
	\end{theorem}

	\subsection{Strategy: heat kernel-regularized \texorpdfstring{$c$}{c}-transform}

Motivated by   regularized $c$-transforms  using  Gibbs kernels $e^{-c(x, y)/\varepsilon}$ \cite{arXiv:2504.05412},  entropic optimal transport \cite{GTJEMS, GTT25}  and Varadhan's formula
\[
\lim_{t \to 0} -t \log p_{t/2}(x, y) = \frac{1}{2} \d^2(x, y)=c(x,y),
\]
 we make use of the following {heat kernel-regularized $c$-transform} \footnote{We are told by Luca Taminini that in  solving the Schr\"odinger equation with the Cole--Hopf transform,  a similar formula will occur.}:
\[
\Lip(X, \d) \ni \psi  \mapsto \Phi_t[\psi](x) = -t \log \int_X e^{\frac{\psi(y)}{t}} p_{t/2}(x, y) \, \d\mm(y).
\]
This approach  allows us to bypass the low regularity of Kantorovich potentials in the non-smooth setting. 

We define the heat kernel regularized Kantorovich functional as
	\begin{equation*}
		\K_t[\psi]:=\int_S \Phi_t[\psi]\,\d\rho.
	\end{equation*}
Similar to \cite{arXiv:2504.05412},  the key in our proof is the strong concavity of the functional $\K_t$.
To achieve this,  we  first derive a local strong concavity estimate using heat kernel estimate,    then globalize the estimate on the support of the source measure using a Boman chain argument for John domains.

Unlike the regularized $c$-transform used in \cite{arXiv:2504.05412},  the existence of the boundary of $Y$ may lead to the failure in our heat kernel regularization argument.  This possibility is ruled out by using the measure concentration property of the heat kernel and by making a careful choice of Lipschitz extension.


\medskip

	\noindent {\bf Organization.}  This paper is structured as follows. In Section~\ref{sec3}, we prove the quantitative stability of Kantorovich potentials on $\mathrm{RCD}(K,N)$ spaces. In Section \ref{sec4}, we establish the stability of optimal transport maps on Alexandrov spaces. The Appendix \ref{A} contains technical lemmas about Poincar\'e inequalities.

	\medskip

\noindent {\bf Acknowledgement.} The authors  thank Nicola Gigli, Jun Kitagawa, Nan Li,  Cyril Letrouit, Quentin M\'erigot, Luca Tamanini for  helpful discussions and  suggestions on the bibliography.

	\section{Stability of Kantorovich potentials}\label{sec3}

	\subsection{Heat kernel estimate}
	We begin by recalling the short-time asymptotic behaviour of the heat kernel on ${\rm RCD}$ spaces. It was first studied by Varadhan on Riemannian manifolds \cite{zbMATH03249197}, and is known as Varadhan formula today. 
	\begin{lemma}\label{lemma0.2}
		Let $p_t(x,y)$ be the heat kernel on an ${\rm RCD}(K,N)$ space $\ms$.  Then
		\begin{equation}
			\lim_{t\rightarrow 0}-t \log p_\frac{t}{2}(x,y)=\frac{1}{2}\d^2(x,y),\quad \text{uniformly in}  \quad x\in S,\, y\in Y.
		\end{equation}
	\end{lemma}
	
	\begin{proof}
		By the heat kernel estimate  \cite[Theorem 1.2] {zbMATH06580546}, for any $\epsilon>0$,  it holds
		\begin{equation}\label{hk0}
			\begin{aligned}
				\frac{1}{C_1(\epsilon)\mm\left(B\big(y, \sqrt{\frac{t}{2}}\big)\right)} &\exp\left(-\frac{\d^2(x, y)}{(4 - \epsilon)t}-C_2(\epsilon)t\right) \\
				\leq p_t(x, y)
				 \leq &
				\frac{C_1(\epsilon)}{\mm\left(B\big(y, \sqrt{\frac{t}{2}}\big)\right)} \exp\left(-\frac{\d^2(x, y)}{(4 + \epsilon)t}+C_2(\epsilon)t\right).
			\end{aligned}
		\end{equation}
		Then we have 
		\begin{equation}\label{hk1}
			\begin{aligned}
				-t\log &\frac{C_1(\epsilon)}{\mm\left(B\big(y, \sqrt{\frac{t}{2}}\big)\right)}+\frac{2\d^2(x, y)}{4+\epsilon}-\frac{C_2(\epsilon)}{2}t^2\\
				\leq &-t \log p_\frac{t}{2}(x,y)
				\leq  t\log C_1(\epsilon)\mm\left(B\big(y, \sqrt{\frac{t}{2}}\big)\right)+\frac{2\d^2(x, y)}{4-\epsilon}+\frac{C_2(\epsilon)}{2}t^2.
			\end{aligned}
		\end{equation}
		
		By Bishop--Gromov inequality \cite[Theorem 2.3]{zbMATH05049052}, we have
		\begin{equation}
		\lim_{t\rightarrow0}t\log \mm\left(B\big(y, \sqrt{\frac{t}{2}}\big)\right)=0,\quad \text{uniformly in } \, y\in Y.
		\end{equation}
Letting $t\rightarrow0$   in \eqref{hk1}, we get the following uniform estimate:
		\begin{equation}
			\frac{2\d^2(x, y)}{4+\epsilon}\leq   \limi_{t\rightarrow0}-t \log p_\frac{t}{2}(x,y)\leq \lims_{t\rightarrow0}-t \log p_\frac{t}{2}(x,y)\leq\frac{2\d^2(x, y)}{4-\epsilon}.
		\end{equation}
		Letting   $\epsilon\rightarrow0$, we prove the lemma.
	\end{proof}
	
	\subsection{Heat kernel regularization}\label{sub}
	To establish the quantitative stability, we adopt a regularization technique similar to those in \cite{MerigotDelalandeChazal2020, arXiv:2504.05412} and \cite{GTJEMS, GTT25} (see also  \cite{ChizatDelalandeVaskevicius2024,CorderoErausquinKlartag2015,MischlerTrevisan2025}).

For $\psi\in\Lip(Y,\d)$,  define 
	\begin{equation*}
	\bar \psi (y)= \sup_{z \in Y} \big\{ \psi(z) - \text{Lip}(\psi) \, \d(z, y) \big\},\quad y\in X.
	\end{equation*}
By MacShane's Lemma,   $\Lip(\bar\psi) =\Lip(\psi)$ and ${\bar\psi}\restr{Y} = \psi \restr Y$. 
	Denote $D=\diam(S\cup Y)$,  $\Lambda_\psi=D+ \Lip(\psi) + 1$ and  define $\psi_*={\bar \psi}-\Lambda_\psi \d(\cdot, Y)$.  
	
	For  $t>0$,  we define the heat kernel regularized $c$-transform as
	\begin{equation*}
\Lip(X, \d)\ni	\varphi\mapsto 	\Phi_t[\varphi](x):=-t \log \int_X e^{\frac{{\varphi}(y)}{t}} p_{\frac{t}{2}}(x, y)\, \d\mm(y),
	\end{equation*}
	and define the heat kernel regularized Kantorovich functional as
	\begin{equation*}
		\K_t[\varphi]:=\int_S \Phi_t[\varphi]\,\d\rho.
	\end{equation*}

	For $\varphi\in\Lip(X,\d)$ and $x\in X$, we associate a probability measure
	\begin{equation*}
		\d\mu_x^t[\varphi](y):=\frac{e^{\frac{\varphi(y)}{t}}p_\frac{t}{2}(x,y)\,\d\mm(y)}{\int_{X} e^{\frac{\varphi(y)}{t}} p_\frac{t}{2}(x,y)\,\d\mm(y)}
	\end{equation*}
and define  $	\mu^t[\varphi]:=\int_X \mu_x^t[\varphi]\, \d \rho(x)$.   This means,   for any $v\in\Lip(X,\d)$,  it holds
	\begin{equation}\label{mu}
	\E_{\mu^t[\varphi]}(v)=\int_X  \E_{\mu_x^t[\varphi]}(v) \,\d \rho(x).	
	\end{equation}

	Next we prove the Sobolev regularity of $x\mapsto \E_{\mu_x^t[\varphi]}(v)$.  We refer to \cite{ zbMATH06303881} for Sobolev calculus on metric measure spaces.  

		\begin{lemma}\label{Lemma3.3}
		For  $t\in \big(0,\frac{D+1}{\sqrt{(N-1)|K|}}\wedge1\big)$, $\psi\in\Lip(Y,\d)$ and  $v\in\Lip(X,\d)$,  there holds $\E_{\mu_x^t[\psi_*]}(v)\in W^{1,2}(X,\d,\mm)$.
	\end{lemma}
	
	\begin{proof}
		Denote by $H_t f(x)=\int_X f(y) p_t(x, y)\,\d \mm(y)$ the heat flow from $f$. Then
		$$\E_{\mu_x^t[\psi_*]}(v)=\frac{\int_{X} v(y)e^{\frac{\psi_*(y)}{t}} p_\frac{t}{2}(x,y)\,\d\mm(y)}{\int_{X} e^{\frac{\psi_*(y)}{t}} p_\frac{t}{2}(x,y)\,\d\mm(y)}=\frac{H_\frac{t}{2}(v e^\frac{\psi_*}{t})(x)}{H_\frac{t}{2}(e^\frac{\psi_*}{t})(x)}.$$
		{\bf Claim}: $v e^\frac{\psi_*}{t}, e^\frac{\psi_*}{t}\in L^2(\mm)\cap L^\infty(\mm)$.
		
		Note that  $|v(y)| \leq \sup_{Y}|v|+ \Lip(v)\diam(Y)+\Lip(v) \d(y, Y)$ and $\psi_*(y)\leq \sup_Y \psi-(D+1)\d(y, Y)$.  So  $v e^\frac{\psi_*}{t}, e^\frac{\psi_*}{t}\in  L^\infty(\mm)$.
		
		 Fix $y_0\in Y$ and denote $B_i=\{y\in X: i\leq\d(y,Y)<i+1\}, i\in \N$.   It holds
		\begin{equation}\label{3.6}
			\begin{aligned}
			\int_X e^\frac{2\psi_*}{t}\d\mm
			&\leq \sum_{i=0}^{+\infty} e^{\frac{2 \sup_Y \psi}{t}}e^\frac{-2(D+1)i}{t}\mm(B_i)\\
			&\leq  \sum_{i=0}^{+\infty} e^{\frac{2 \sup_Y \psi}{t}}e^\frac{-2(D+1)i}{t}\mm(B(y_0, \diam(Y)+i+1)).		
			\end{aligned}
		\end{equation}

		Without loss of generality,   we may assume $K<0$.  By Bishop--Gromov inequality \cite[Theorem 2.3]{zbMATH05049052}, we have
		\begin{equation*}
		\mm\big(B(y_0, \diam(Y)+i+1)\big)\leq c_1\int_0^{\diam(Y)+i+1}\sinh^{N-1}(\sqrt{\frac{-K}{N-1}}s)\,\d s
		\end{equation*}
		where $c_1=\frac{\mm(B(y_0, \diam(Y)))}{\int_0^{\diam(Y)}\sinh^{N-1}(\sqrt{\frac{-K}{N-1}}s)\d s}$.
		Note that $\sinh^{N-1}(\sqrt{\frac{-K}{N-1}}s)\leq \frac{e^{\sqrt{(N-1)|K|}s}}{2^{N-1}}$, then 
		\begin{equation}\label{2.7}
		\mm\big(B(y_0, \diam(Y)+i+1)\big)\leq c_1 \int_0^{\diam(Y)+i+1}\frac{e^{\sqrt{(N-1)|K|}s}}{2^{N-1}}\,\d s\leq c_2 e^{\sqrt{(N-1)|K|}i}
		\end{equation}
		where $c_2=c_1 \frac{e^{\sqrt{(N-1)|K|}(\diam(Y)+1)}}{2^{N-1}\sqrt{{(N-1)}|K|}}$. 

		For $t<\frac{D+1}{\sqrt{(N-1)|K|}}$,  it holds $\sqrt{(N-1)|K|}-\frac{2(D+1)}{t}<-\frac{D+1}{t}$,  so by \eqref{3.6} and \eqref{2.7},
		\begin{equation}
			\int_X e^\frac{2\psi_*}{t}\,\d\mm\leq \sum_{i=0}^{+\infty} e^{\frac{2 \sup_Y \psi}{t}}c_2 e^\frac{-2(D+1)i}{t}e^{\sqrt{(N-1)|K|}i}\leq  \sum_{i=0}^{+\infty} e^{\frac{2 \sup_Y \psi}{t}}c_2 e^\frac{-(D+1)i}{t}<+\infty.
		\end{equation}
		Similarly, we can prove
			$\int_X v^2 e^\frac{2\psi_*}{t}\,\d\mm<+\infty$ and we prove the claim.
		\medskip
		
		By the regularization of heat flow  \cite[Theorem 6.5]{zbMATH06303881}, we have $H_\frac{t}{2}(v e^\frac{\psi_*}{t}), H_\frac{t}{2}(e^\frac{\psi_*}{t})\in W^{1,2}(X,\d,\mm)\cap \Lip(X,\d)\cap L^\infty(\mm)$. Moreover, by the heat kernel estimate \eqref{hk0},
		\begin{equation}\label{3.15}
			\begin{aligned}
				&H_\frac{t}{2}(e^\frac{\psi_*}{t})(x)
				\geq \int_{Y} e^{\frac{\psi(y)}{t}} p_\frac{t}{2}(x,y)\,\d\mm(y)>0.
			\end{aligned}
		\end{equation}
	So $\frac{1}{H_\frac{t}{2}(e^\frac{\psi_*}{t})}\in W^{1,2}(X,\d,\mm)\cap L^\infty(\mm)$ and by chain rule  (see \cite[Theorem 4.3.3]{zbMATH07183685}), we know the function  $x\mapsto \E_{\mu_x^t[\psi_*]}(v)\in W^{1,2}(X,\d,\mm)$.
	\end{proof}
		
\medskip

	\begin{lemma}\label{lemma0.3}
		For any $v\in\Lip(X,\d)$ and $\psi\in \Lip(Y,  \d)$, we have 
		\begin{itemize}
		\item  $\frac{\d}{\d s}\restr{s=0}{\K}_t[\psi_*+sv]=-  \E_{\mu^t[\psi_*]}(v)$;
		\item $\frac{\d^2}{\d s^2}\restr{s=0}{\K}_t[\psi_*+sv]=-\frac{1}{t}\int_S {\rm{Var}}_{ \mu_x^t[\psi_*]}(v)\,\d\rho(x)$.
		
		\end{itemize}
		
	For abbreviation,  we  can write
		$$\nabla {\K}_t[\psi_*]=- \mu^t[\psi_*],\quad  \langle D^2{\K}_t[\psi_*]v, v\rangle=-\frac{1}{t}\int_S {\rm{Var}}_{ \mu_x^t[\psi_*]}(v)\,\d\rho(x).$$
	\end{lemma}
	
	\begin{proof}
		Let $|s| \leq 1$.
		By direct computation,
		\begin{equation*}
			\frac{\d}{\d s}\Phi_t[\psi_*+sv]=-\frac{\int_{X} v(y)e^{\frac{\psi_*(y)+sv(y)}{t}} p_\frac{t}{2}(x,y)\,\d\mm(y)}{\int_{X} e^{\frac{\psi_*(y)+sv(y)}{t}} p_\frac{t}{2}(x,y)\,\d\mm(y)}.
		\end{equation*}
		Similar to Lemma \ref{Lemma3.3},  we can see that
		$\left|\frac{\d}{\d s}\Phi_t[\psi_*+sv]\right|$
		 is uniformly bounded in $s$.
		 
		By differentiating under the integral defining ${\K}_t[\psi_*]$ and letting $s=0$, we have
		\begin{equation*}
			\frac{\d}{\d s}\restr{s=0}{\K}_t[\psi_*+sv]=\int_S \frac{\d}{\d s}\restr{s=0}\Phi_t[\psi_*+sv]\,\d\rho(x)=- \E_{\mu^t[\psi_*]}(v).
		\end{equation*}
		
		Similarly,    we can prove
		\begin{equation*}
			\begin{aligned}
				&\frac{\d^2}{\d s^2}\Phi_t[\psi_*+sv]\\
				=&-\frac{1}{t}\left(\frac{\int_{X} v^2(y)e^{\frac{\psi_*(y)+sv(y)}{t}}p_\frac{t}{2}(x,y)\,\d\mm(y)}{\int_{X} e^{\frac{\psi_*(y)+sv(y)}{t}}p_\frac{t}{2}(x,y)\,\d\mm(y)}
				-\left|\frac{\int_{X} v(y) e^{\frac{\psi_*(y)+sv(y)}{t}} p_\frac{t}{2}(x,y)\,\d\mm(y)}{\int_{X} e^{\frac{\psi_*(y)+sv(y)}{t}} p_\frac{t}{2}(x,y)\,\d\mm(y)}\right|^2\right)
			\end{aligned}
		\end{equation*}
		and
		\begin{equation*}
			\frac{\d^2}{\d s^2}\restr{s=0}{\K}_t[\psi_*+sv]=\int_S \frac{\d^2}{\d s^2}\restr{s=0}\Phi_t[\psi_*+sv]\,\d\rho(x)=-\frac{1}{t}\int_S {\rm{Var}}_{ \mu_x^t[\psi_*]}(v)\,\d\rho(x).
		\end{equation*}
	\end{proof}

	\subsection{Gradient estimate}
	To leverage the Hessian formula in Lemma \ref{lemma0.3}, we need to bound the variance term from below; this amounts to proving strong concavity of the regularised Kantorovich functional.  We achieve this by establishing a gradient estimate for the marginal density $x \mapsto \E_{\mu_x^t[\psi_*]}(v)$.

The following auxiliary lemma will be used throughout this section. 	 We refer to \cite{G-N,  zbMATH07183685} for a comprehensive introduction to $L^2$-normed  tangent module $L^2(TX)$.   By \cite[Proposition 2.9]{H-R} and \cite[Theorem 1.4.11]{G-N}, elements of $L^2(TX)$ admit local-coordinate representations.  Readers unfamiliar with non-smooth calculus may safely interpret the proof in the language of Riemannian geometry.

		\begin{lemma}\label{grad}
		Let $f\in L^2(\mm)\cap L^\infty(\mm)$,  and $g\in L^0(X\times X)$ be with $g(\cdot, y),  g(x, \cdot) \in W^{1,2}\ms$ for  $\mm$-a.e. $x, y\in X$. Then  
	$
		\int_X f(y)  g(x,y)\,\d\mm(y)\in W^{1,2}\ms$  and for any $\varphi \in W^{1,2}\ms$, we have  
\begin{equation}\label{l2.4}
\int \langle \nabla_x \int_X f(y)  g(x,y)\,\d\mm(y),\varphi(x) \rangle\,\d\mm(x)=\int_X f(y)\int \langle  \nabla g(\cdot,   y), \varphi\rangle\,\d\mm\,\d\mm(y)
\end{equation}
and
	\begin{equation}\label{17}
		\left|\nabla_x \int_X f(y)  g(x,y)\,\d\mm(y)\right|\leq \int_X |f(y)| |\nabla_x g(x,y)|\,\d\mm(y).
		\end{equation}
		
In particular,   we can write
		\begin{equation}
			\int_X f(y) \nabla_x g(x,y)\,\d\mm(y)=\nabla_x \int_X f(y)  g(x,y)\,\d\mm(y)\in L^2(TX).
		\end{equation}
	\end{lemma}
	\begin{proof}
	 For $\varphi\in	{\rm TestF}:=\big\{\phi\in \Lip(X, \d)\cap L^\infty(\mm)\cap \D(\Delta): \Delta \phi \in W^{1,2}(X,\d,\mm)\cap L^\infty(\mm)\}$,  by integration by parts and Fubini theorem, we get
		\begin{equation}
			\begin{aligned}
			&\int_X \Delta \varphi(x) \left(\int_X f(y)g(x,y)\,\d\mm(y)\right)\,\d\mm(x)\\
			=&\int_X f(y)\int_X g(x,y)\Delta \varphi(x)\,\d\mm(x)\,\d\mm(y)\\
			=&-\int_X f(y)\int_S \langle  \nabla_x g(x,y),\nabla \varphi(x)\rangle\,\d\mm(x)\,\d\mm(y).
			\end{aligned}
		\end{equation}
By density of ${\rm TestF}$ (cf.  \cite[\S 3.2]{G-N}) and the Riesz representation theorem for Hilbert module (cf.\cite[Theorem 3.2.14,  Example 3.2.15 ]{zbMATH07183685}), we prove the lemma.
			\end{proof}

\medskip

	\begin{lemma}\label{prop0.5}
		For any $\psi\in \Lip(Y,  \d)$,  it holds
		\begin{equation}\label{2.14}
			|\nabla_x \E_{\mu_x^t[\psi_*]}(v)|^2\leq {\rm{Var}}_{ \mu_x^t[\psi_*]}(v) {\rm{Var}}_{ \mu_x^t[\psi_*]}\big(\nabla_x \log p_\frac{t}{2}(x,\cdot)\big)
		\end{equation}
		for $\mm$-a.e. $x\in S$, where   $${\rm{Var}}_{ \mu_x^t[\psi_*]}\big(\nabla_x \log p_\frac{t}{2}(x,\cdot)\big):=\E_{\mu_x^t[\psi_*]}\left(\big|\nabla_x \log p_\frac{t}{2}(x,\cdot)-	\E_{\mu_x^t[\psi_*]}(\nabla_x \log p_\frac{t}{2}(x,\cdot))\big|^2\right)$$
		and
		\begin{equation*}
			\begin{aligned}
				\E_{\mu_x^t[\psi_*]}\big(\nabla_x \log p_\frac{t}{2}(x,\cdot)\big)
				= \frac{\int_X e^\frac{\psi_*(z)}{t} \nabla_x p_\frac{t}{2}(x,z)\,\d\mm(z)}{\int_X e^\frac{\psi_*(z)}{t} p_\frac{t}{2}(x,z)\,\d\mm(z)}.
			\end{aligned}
		\end{equation*}
	\end{lemma}
	
	\begin{proof}
	First of all,  by \cite{zbMATH06580546},   $p_s(\cdot,y)\in W^{1,2}(X,\d,\mm)$ for $\mm$-a.e. $y\in X$ so all the formulas above are well-defined.  	By Lemma \ref{Lemma3.3}, Lemma \ref{grad} and the  chain rule for  $L^2$-normed  modules (cf. \cite[Theorem 4.3.3]{zbMATH07183685}),  we have
		\begin{eqnarray*}
		&&\nabla_x \E_{\mu_x^t[\psi_*]}(v)=\frac{H_\frac{t}{2}(e^\frac{\psi_*}{t})\nabla H_\frac{t}{2}(v e^\frac{\psi_*}{t})-H_\frac{t}{2}(v e^\frac{\psi_*}{t})\nabla H_\frac{t}{2}(e^\frac{\psi_*}{t})}{\big(H_\frac{t}{2}(e^\frac{\psi_*}{t})\big)^2}(x)\\
		&=& \int_{X} v(y)\nabla_x \log p_\frac{t}{2}(x,y)\,\d\mu_x^t[\psi_*](y)-\E_{\mu_x^t[\psi_*]}(v)\int_{X} \nabla_x \log p_\frac{t}{2}(x,y)\,\d\mu_x^t[\psi_*](y)\\ &=& \int_{X} \big(v(y)-\E_{\mu_x^t[\psi_*]}(v)\big)\left(\nabla_x \log p_\frac{t}{2}(x,y)-	\E_{\mu_x^t[\psi_*]}\big(\nabla_x \log p_\frac{t}{2}(x,\cdot)\big)\right)\,\d\mu_x^t[\psi_*](y),
		\end{eqnarray*}
where in the last equality we use the identity
\begin{eqnarray*}
&&\int_{X} \big(v(y)-\E_{\mu_x^t[\psi_*]}(v)\big)\E_{\mu_x^t[\psi_*]}\big(\nabla_x \log p_\frac{t}{2}(x,\cdot)\big)\,\d\mu_x^t[\psi_*](y)\\ &=&\E_{\mu_x^t[\psi_*]}\big(\nabla_x \log p_\frac{t}{2}(x,\cdot)\big)\int_{X} \big(v(y)-\E_{\mu_x^t[\psi_*]}(v)\big)\,\d\mu_x^t[\psi_*](y)=0.
\end{eqnarray*}	

By Lemma \ref{grad} and  H\"older inequality we get  \eqref{2.14}.
	\end{proof}
	
	\begin{lemma}\label{Lemma3.6}
		The following identity holds: 
		\begin{equation}
			{\rm{Var}}_{ \mu_x^t[\psi_*]}\big(\nabla_x \log p_\frac{t}{2}(x,\cdot)\big)= \Delta_x \log \left(H_\frac{t}{2}(e^\frac{\psi_*}{t})(x)\right) - \mathbb{E}_{\mu_x^t[\psi_*]}\big(\Delta_x \log p_{\frac{t}{2}}(x, \cdot)\big).
		\end{equation}
	\end{lemma}
	\begin{proof}
		By chain rule of Laplacian (cf. \cite[Theorem 5.2.3]{zbMATH07183685}), it holds
	\begin{equation}\label{De}
	\Delta\log u=\frac{\Delta u}{u}-|\nabla \log u|^2,
	\end{equation}
	for $u=p_\frac{t}{2}(x,y)$ and $H_\frac{t}{2}(e^\frac{b}{t})$. By Lemma~ \ref{grad},  we have
	\begin{equation}
		\begin{aligned}
		&{\rm{Var}}_{ \mu_x^t[\psi_*]}\big(\nabla_x \log p_\frac{t}{2}(x,\cdot)\big)\\
		=&\int_X |\nabla_x \log p_\frac{t}{2}(x,y)|^2\,\d\mu_x^t[\psi_*](y)-\left|\int_X \nabla_x \log p_\frac{t}{2}(x,y)\,\d\mu_x^t[\psi_*](y)\right|^2\\
		\mathop{=}^{\eqref{De}}&\int_X \frac{\Delta_x p_\frac{t}{2}(x,y)}{ p_\frac{t}{2}(x,y)}\,\d\mu_x^t[\psi_*](y)-\mathbb{E}_{\mu_x^t[\psi_*]}\big(\Delta_x \log p_{\frac{t}{2}}(x, \cdot)\big)\\
		&-\left|\int_X \nabla_x \log p_\frac{t}{2}(x,y)\,\d\mu_x^t[\psi_*](y)\right|^2\\
		=&\frac{\Delta_x H_\frac{t}{2}(e^\frac{\psi_*}{t})(x)}{H_\frac{t}{2}(e^\frac{\psi_*}{t})(x)}-|\nabla \log H_\frac{t}{2}(e^\frac{\psi_*}{t})(x)|^2- \mathbb{E}_{\mu_x^t[\psi_*]}\big(\Delta_x \log p_{\frac{t}{2}}(x, \cdot)\big)\\
		\mathop{=}^{\eqref{De}}&\Delta_x \log \left(H_\frac{t}{2}(e^\frac{\psi_*}{t})(x)\right)  - \mathbb{E}_{\mu_x^t[\psi_*]}\big(\Delta_x \log p_{\frac{t}{2}}(x, \cdot)\big)
		\end{aligned}
	\end{equation}
	which is the thesis.
	\end{proof}

	\begin{lemma}\label{p3.7}
		For $t>0$ small enough,  it holds 
		\begin{equation}
			J_t(x):=\mathbb{E}_{\mu_x^t[\psi_*]}\big(|\nabla_x \log p_\frac{t}{2}(x,\cdot)|^2\big)\leq \frac{C_1(K,N, \Lambda_\psi)}{t^2}.
		\end{equation} 
	\end{lemma}
	
	\begin{proof}
	Integrating the following Li--Yau type estimate  (cf. \cite[Theorem 1.2]{zbMATH06441788})
		\begin{equation}\label{Li-Yau}
			|\nabla_x \log p_\frac{t}{2}(x,y)|^2\leq e^{-\frac{Kt}{3}}\frac{\Delta_x p_\frac{t}{2}(x,y)}{p_\frac{t}{2}(x,y)}+\frac{NK}{3}\frac{e^{-\frac{2Kt}{3}}}{1-e^{-\frac{Kt}{3}}}
		\end{equation}
with respect to $\mu_x^t[\psi_*]$, we obtain
		\begin{equation}\label{0.12}
			J_t(x)\leq e^{-\frac{Kt}{3}}\int_{X} \frac{\Delta_x p_\frac{t}{2}(x,y)}{p_\frac{t}{2}(x,y)}\,\d \mu_x^t[\psi_*]+\frac{NK}{3}\frac{e^{-\frac{2Kt}{3}}}{1-e^{-\frac{Kt}{3}}}.
		\end{equation}
		 By symmetry of the heat kernel,    we have $$\int_{X} \Delta_x p_\frac{t}{2}(x,y) e^{\frac{\psi_*(y)}{t}}\,\d\mm(y)=\int_{X} \Delta_y p_\frac{t}{2}(x,y) e^{\frac{\psi_*(y)}{t}}\,\d\mm(y).$$ Then by integration by parts formula,
		\begin{equation}\label{0.14}
			\begin{aligned}
				\int_{X} \frac{\Delta_x p_\frac{t}{2}(x,y)}{p_\frac{t}{2}(x,y)}\,\d \mu_x^t[\psi_*](y)=&\frac{\int_{X} \Delta_x p_\frac{t}{2}(x,y) e^{\frac{\psi_*(y)}{t}}\,\d\mm(y)}{\int_{X} e^{\frac{\psi_*(y)}{t}} p_\frac{t}{2}(x,y)\,\d\mm(y)}\\
				=&\frac{\int_{X} \Delta_y p_\frac{t}{2}(x,y) e^{\frac{\psi_*(y)}{t}}\,\d\mm(y)}{\int_{X} e^{\frac{\psi_*(y)}{t}} p_\frac{t}{2}(x,y)\,\d\mm(y)}\\
				=&\frac{-\int_{X} \langle\nabla_y p_\frac{t}{2}(x,y), \nabla_y e^{\frac{\psi_*(y)}{t}}\rangle \,\d\mm(y)}{\int_{X} e^{\frac{\psi_*(y)}{t}} p_\frac{t}{2}(x,y)\,\d\mm(y)}\\
				\leq&\int_{X} |\nabla_y \log p_\frac{t}{2}(x,y)| \left| \frac{\nabla_y\bar \psi(y)-\Lambda_\psi\nabla_y  \d(y,Y)}{t}\right|\,\d\mu_x^t[\psi_*](y)\\
				\leq&\frac{2\Lambda_\psi}{t} \left( \underbrace{\int_X |\nabla_y \log p_\frac{t}{2}(x,y)|^2 \d \mu_x^t[\psi_*](y)}_{:=\tilde{J}_t(x)} \right)^\frac 12.
			\end{aligned}
		\end{equation}

		Similarly, we can  prove
		\begin{equation}\label{2.26}
		\bar{J}_t(x)
		\leq e^{-\frac{Kt}{3}}\frac{2\Lambda_\psi}{t}\bar{J}_t^\frac{1}{2}(x)+\frac{NK}{3}\frac{e^{-\frac{2Kt}{3}}}{1-e^{-\frac{Kt}{3}}}.
		\end{equation}
Since $1-e^{-\frac{Kt}{3}}=\frac{K}{3}t+o(t)$,  we have
			$\bar{J}_t(x)\lesssim \frac{1}{t^2}$. Then  \eqref{0.12} and \eqref{0.14} implies
		\begin{equation}
			\begin{aligned}
				J_t(x)\leq e^{-\frac{Kt}{3}}\frac{2\Lambda_\psi}{t}\bar{J}_t^\frac{1}{2}(x)+\frac{NK}{3}\frac{e^{-\frac{2Kt}{3}}}{1-e^{-\frac{Kt}{3}}}\leq \frac{C_1(K,N, \Lambda_\psi)}{t^2}.
			\end{aligned}
		\end{equation}

\end{proof}

	\begin{lemma}\label{prop3.8}
		Let $x_0\in S$, $0<r_0\leq1$ and $B=B(x_0,r_0)\subseteq S$. For $t>0$ small enough,  it holds
		\begin{equation}\label{2.24}
		 \int_B {\rm{Var}}_{ \mu_x^t[\psi_*]}\big(\nabla_x \log p_\frac{t}{2}(x,\cdot)\big) \,\d\rho(x) \leq \frac{C_2\mm(B(x_0,2r_0))}{r_0t},
		\end{equation}
		where $C_2$ depends on $K,N,\Lambda_\psi,a_2$.
		\end{lemma}
		\begin{proof}
		Take a cut-off function $\eta$ satisfying $\eta \equiv 1$ on $B$, $\eta \equiv 0$ on $X\setminus B(x_0,2r_0)$, $0 \le \eta \le 1$, and $|\nabla \eta| \leq \frac{10}{r_0}$. Recall that $a_1\mm\restr{S} \leq \rho\leq a_2\mm\restr{S}$, by Lemma \ref{Lemma3.6}, we have  
		\begin{equation}\label{J_t^cen}
			\begin{aligned}
			&\int_B {\rm{Var}}_{ \mu_x^t[\psi_*]}\big(\nabla_x \log p_\frac{t}{2}(x,\cdot)\big)  \,\d\rho(x)\\
			  \le & a_2 \int_X \eta(x) {\rm{Var}}_{ \mu_x^t[\psi_*]}\big(\nabla_x \log p_\frac{t}{2}(x,\cdot)\big)  \,\d\mm(x) \\
			 = &a_2\left(\int_X \eta(x) \Delta_x \log \left(H_\frac{t}{2}(e^\frac{\psi_*}{t})(x)\right) \,\d\mm(x)
			- \int_X \eta(x) \mathbb{E}_{\mu_x^t[\psi_*]} \big(\Delta_x \log p_{\frac{t}{2}}(x, \cdot)\big) \,\d\mm(x)\right) \\
			 := & a_2(I_1 + I_2).
		\end{aligned}
		\end{equation}
		
		For $I_1$,  by integration by parts formula and H\"older inequality, 
		\begin{equation}
			\begin{aligned}
			I_1
			&=-\int_X \langle\nabla\eta, \nabla \log \big(H_\frac{t}{2}(e^\frac{\psi_*}{t})\big)\rangle\,\d\mm\\
			&\leq\left( \int_{B(x_0,2r_0)\setminus B}|\nabla \eta|^2\,\d\mm\right)^\frac{1}{2}\left( \int_{B(x_0,2r_0)\setminus B} \left| \nabla\log \big(H_\frac{t}{2}(e^\frac{\psi_*}{t})\big)\right|^2\,\d\mm\right)^\frac{1}{2}.
			\end{aligned}
		\end{equation}
By Lemma \ref{grad}, it holds
		\begin{equation}
			\left|\nabla\log \big(H_\frac{t}{2}(e^\frac{\psi_*}{t})\big)\right|(x)\leq \int_X |\nabla_x \log p_\frac{t}{2}(x,y)| \,\d \mu_x^t[\psi_*](y).
		\end{equation}
		  Combining with Lemma \ref{p3.7}, we get
		\begin{equation}\label{3.37}
		I_1\leq \frac{10\sqrt{C_1}\mm(B(x_0,2r_0))}{r_0t}.
		\end{equation}
		\medskip
		
		For $I_2$, by \eqref{De} and the Li--Yau type estimate \eqref{Li-Yau}, we have
		\begin{equation}
			\begin{aligned}
				-\mathbb{E}_{\mu_x^t[\psi_*]}(\Delta_x \log p_{\frac{t}{2}}(x, \cdot))&=\mathbb{E}_{\mu_x^t[\psi_*]}\left(|\nabla_x \log p_{\frac{t}{2}}(x,\cdot)|^2-\frac{\Delta_x p_{\frac{t}{2}}(x,\cdot)}{p_{\frac{t}{2}}(x,\cdot)}\right)\\
				&\leq (e^{-\frac{Kt}{3}}-1)\int_X \frac{\Delta_x p_\frac{t}{2}(x,y)}{p_\frac{t}{2}(x,y)}\,\d \mu_x^t[\psi_*](y)+\frac{NK}{3}\frac{e^{-\frac{2Kt}{3}}}{1-e^{-\frac{Kt}{3}}}.
			\end{aligned}
		\end{equation}
		By \eqref{0.14} and \eqref{2.26}, for $t$ small enough, there is   $c_1=c_1(K, N, \Lambda_\psi)$ so that
		\begin{equation*}
			-\mathbb{E}_{\mu_x^t[\psi_*]}(\Delta_x \log p_{\frac{t}{2}}(x, y))\leq \left|e^{-\frac{Kt}{3}}-1\right|\frac{2\Lambda_\psi}{t}\bar{J}_t^\frac{1}{2}(x)+\frac{NK}{3}\frac{e^{-\frac{2Kt}{3}}}{1-e^{-\frac{Kt}{3}}}\leq \frac{c_1}{t}.
		\end{equation*}
So
				$I_2
			\leq  \frac{c_1\mm(B(x_0,2r_0))}{t}$.
Combining with \eqref{J_t^cen}, \eqref{3.37}, we obtain \eqref{2.24}
		\end{proof}
		
		\medskip
		
		\begin{proposition}\label{3.9}
		It holds that 
		\begin{equation}\label{231}
		\int_B 	|\nabla_x \E_{\mu_x^t[\psi_*]}(v)|\,\d\rho(x)\leq \frac{\sqrt{C_2\mm(B(x_0,2r_0))}}{\sqrt{r_0
		t}}\left(\int_B {\rm{Var}}_{ \mu_x^t[\psi_*]}(v)\,\d\rho(x)\right)^\frac{1}{2}.
		\end{equation}

		\end{proposition}
		\begin{proof}
	It follows from Lemma \ref{prop0.5}, Lemma \ref{prop3.8}, and H\"{o}lder inequality.
		\end{proof}

	\subsection{Global concavity estimate}
	We  improve  the estimate \eqref{231} in Proposition\ref{3.9} to a global one. The proof is similar to  \cite[Lemma 3.3]{arXiv:2411.04908},  which involves two key lemmas in Appendix \ref{A}.

	\begin{proposition}\label{2.11}
	It holds that
	\begin{equation}
	\int_{S}\big| \E_{\mu_x^t[\psi_*]}(v)-\E_{\mu^t[\psi_*]}(v))\big|\,\d\rho(x)\leq  \frac{\kappa}{\sqrt{t}}\left(\int_S {\rm{Var}}_{ \mu_x^t[\psi_*]}(v)\,\d\rho(x)\right)^\frac{1}{2},
	\end{equation}
	where $\mu^t[\psi_*]$ is defined as in \eqref{mu} and $\kappa$ depends on $K,N,\Lambda_\psi,a_1,a_2, S$.
	\end{proposition}
	\begin{proof}
		Since $a_1\mm\restr{S} \leq \rho\leq a_2\mm\restr{S}$, by \cite[Corollary 2.4]{zbMATH05049052}, $\rho$ is a doubling measure. Moreover, 
		since $S$ is a John domain, by \cite[Proposition 3.7]{arXiv:2504.05412}, $\rho$ satisfies the Boman chain condition (see Definition \ref{Boman chain condition}) and we can choose a covering $\mathcal{F}$, such that for any $B\in \mathcal{F}$, $r_B\leq 1$. Then 
		\begin{equation}
			\begin{aligned}
				&\int_{S}\big| \E_{\mu_x^t[\psi_*]}(v)-\E_{\mu^t[\psi_*]}(v))\big|\,\d\rho(x)\\
				\mathop{\leq}^{\text {Lemma \ref{3.1}}} &C_4 \sum_{B\in\mathcal{F}}\rho(B)\int_B \big| \E_{\mu_x^t[\psi_*]}(v)-{\rho(B)}^{-1}\E_{\mu^t[\psi_*]}(v))\big|\,\d\rho_B(x)\\
				\mathop{\leq }^{\text{Lemma \ref{p3.10}}}&C_4 \sum_{B\in\mathcal{F}}\rho(B)C_3 r_B\int_B \big|\nabla_x \E_{\mu_x^t[\psi_*]}(v)\big|\,\d\rho_B(x)\\
			\mathop{\leq}^{\text{Proposition \ref{3.9}} }&  \frac{\sqrt{C_2}C_3C_4}{\sqrt{t}}\sum_{B\in\mathcal{F}}\sqrt{\mm(B_{2r_B})}\sqrt{r_B}\left(\int_B {\rm{Var}}_{ \mu_x^t[\psi_*]}(v)\,\d\rho(x)\right)^\frac{1}{2}\\
				\mathop{\leq}^{\text{H\"{o}lder}~~}&\frac{\sqrt{C_2}C_3C_4}{\sqrt{t}}\left(\sum_{B\in\mathcal{F}}\mm(B_{2r_B})\right)^\frac{1}{2}\left(\sum_{B\in\mathcal{F}}\int_B {\rm{Var}}_{ \mu_x^t[\psi_*]}(v)\,\d\rho(x)\right)^\frac{1}{2}\\
					\overset{\text{Boman chain~}}{\leq}&\frac{\kappa}{\sqrt{t}}\left(\int_S {\rm{Var}}_{ \mu_x^t[\psi_*]}(v)\,\d\rho(x)\right)^\frac{1}{2},
			\end{aligned}
		\end{equation}
		which is the thesis.
	\end{proof}


	\subsection{Stability estimate for Kantorovich functionals}

	\begin{proposition}\label{0.30}
	Let $\varphi, \psi\in \Lip(Y, \d)$ and $t>0$.  Then
		\begin{equation}
			\begin{aligned}
				&\int_{S}\left|\Big(\Phi_t[\varphi_{*}](x)-\K_t[\varphi_{*}]\Big)-\Big(\Phi_t(\psi_{*})(x)-\K_t[\psi_{*}]\Big)\right|\,\d\rho(x)\\
				\leq &C_5 \left| \E_{\mu^t[\varphi_{*}]-\mu^t[\psi_{*}]}(\varphi_{*}-\psi_{*})\right|^\frac{1}{2},
			\end{aligned}
		\end{equation}
		where $C_5$ depends on $K,N,a_1,a_2,\diam(S\cup Y),S,  \Lip(\varphi),  \Lip(\psi)$.
	\end{proposition}
	
	\begin{proof}

		Denote $v=\varphi_{*}-\psi_{*}$,  and $ \phi_{s}=\psi_{*}+sv$ for $0\leq s\leq1$. 
		By Lemma \ref{lemma0.3}, 
		\begin{equation}
			\begin{aligned}
				\frac{\d}{\d s}\Big({\Phi}_t[\phi_{s}]-\K_t[\phi_{s}]\Big)=&-\left(\E_{\mu_x^t[\phi_{s}]}(v)-\E_{\mu^t[\phi_{s}]}(v) \right).
			\end{aligned}
		\end{equation}
		
		By Proposition \ref{2.11},  we have  
	    \begin{equation*}
	    	\begin{aligned}
	    		&\int_{S}\left|\Big(\Phi_t[\varphi_{*}](x)-\K_t[\varphi_{*}]\Big)-\Big(\Phi_t[\psi_{*}](x)-\K_t[\psi_{*}]\Big)\right|\,\d\rho(x)\\
	    		\leq &\int_0^1\int_S \left| \frac{\d}{\d s}\Big({\Phi}_t[\phi_{s}]-\K_t[\phi_{s}]\Big)\right|\,\d\rho(x)\,\d s
	    		= \int_0^1\int_S \left|\E_{\mu_x^t[\phi_{s}]}(v)-\E_{\mu^t[\phi_{s}]}(v) \right|\,\d\rho(x)\,\d s\\
	    		\overset{*}{\leq} &\kappa\left(-\int_0^1\langle D^2{\K}_t[\phi_{s}]v, v\rangle\,\d s\right)^\frac{1}{2}\\
	    		= &\kappa\left|\langle \nabla {\K}_t[\varphi_{*}]-\nabla {\K}_t[\psi_{*}], \varphi_{*}-\psi_{*}\rangle\right|^\frac{1}{2}
	    		=\kappa\left| \E_{\mu^t[\varphi_{*}]-\mu^t[\psi_{*}]}(\varphi_{*}-\psi_{*})\right|^\frac{1}{2},
	    	\end{aligned}
	    \end{equation*}
	       where $(*)$ follows from Proposition \ref{2.11},  Lemma \ref{lemma0.3} and H\"older inequality, while the last two equalities follow from  Lemma \ref{lemma0.3}.  This completes the proof.
	\end{proof}
	
	\subsection{Passing to the limit}
	In the following lemma we pass to the limit of $t$ to recover the  Kantorovich potentials.   We remark that,  for $\psi\in C_b(X)$,  this asymptotic formula has been proved by Gigli--Tamanini--Trevisan \cite[Proposition 5.2]{GTT25} in $\rcd$ spaces.
	\begin{lemma}\label{l0.20}
		For any $\psi\in \Lip(Y, \d)$,  we have 
		\begin{equation}\label{2.38}
			\lim_{t\rightarrow0}{\Phi}_t[\psi_{*}](x)=\lim_{t\rightarrow0}-t\log \int_X e^{\frac{\psi_*(y)}{t}} p_\frac{t}{2}(x,y)\,\d\mm(y)=\psi^c(x),
		\end{equation}
		for $\rho$-a.e. $x\in S$, where $\psi^c(x)=\inf_{y\in Y}\{c(x,y)-\psi(y)\}$.   
	\end{lemma}
	\begin{proof}
	For  $\delta>0$ and $Y_\delta:=\{y\in X: \d(y, Y)\leq \delta\}$,  we denote
	\begin{equation}
		\int_X e^{\frac{\psi_{*}(y)}{t}} p_\frac{t}{2}(x,y)\,\d\mm(y)=\underbrace{\int_{Y_\delta} e^{\frac{\psi_{*}(y)}{t}} p_\frac{t}{2}(x,y)\,\d\mm(y)}_{I_1}+\underbrace{\int_{X\setminus Y_\delta} e^{\frac{\psi_{*}(y)}{t}} p_\frac{t}{2}(x,y)\,\d\mm(y)}_{I_2}.
	\end{equation}
	
	{\bf Step 1:} 
Recall that $\psi_{*}(y)=\bar \psi (y)-\Lambda_\psi\d(y,Y)$, $\Lambda_\psi=D+ \Lip(\psi) + 1,  D=\diam(S\cup Y)$.   We have
	\begin{equation}\label{ext}
		\begin{aligned}
			\frac{I_2}{I_1}
			\leq& \frac{e^{\frac{\sup_Y \psi_i-(D+1)\delta}{t}}\int_{X} p_\frac{t}{2}(x,y)\,\d\mm(y)}{e^{\frac{\inf_Y \psi_i}{t}}\int_{Y}  p_\frac{t}{2}(x,y)\,\d\mm(y)}.
		\end{aligned}
	\end{equation}
Combining with the stochastic completeness of the heat flow and the heat kernel estimate \eqref{3.15},    we can find $\delta>0$ large enough such that
 \begin{equation}\label{2.40}
     	\frac{I_2}{I_1}\lesssim  e^{-\frac{D^2}{t}}~~~~\text{as}~~t\to 0.
     \end{equation}
	
	\medskip
	
	{\bf Step 2:} We claim that
	\begin{equation}\label{3.61}
	\inf_{y\in Y_\delta\setminus Y}\{c(x,y)-\psi_{*}(y)\}>\inf_{y\in Y}\{c(x,y)-\psi(y)\}=\psi^c(x).
	\end{equation}
	
	Since $Y$ is compact,  there exists $y_1\in \overline {Y_\delta\setminus Y}$ and $y_2\in Y$,  such that 
	\begin{equation*}
	c(x,y_1)-\psi_{*}(y_1)=\inf_{y\in Y_\delta\setminus Y}\{c(x,y)-\psi_{*}(y)\},~~~\d(y_1,y_2)=\d(y_1,Y)=:s>0.
	\end{equation*}
 Then \eqref{3.61} follows from the following estimate
	\begin{equation}
	\begin{aligned}
	&c(x,y_1)-\psi_{*}(y_1)- \big(c(x,y_2)-\psi(y_2)\big),\\
	\geq & \frac{1}{2}\big(\d(x,y_2)-\d(y_1,y_2)\big)^2-\frac{1}{2}\d^2(x,y_2)-\Lip(\bar \psi)\d(y_1,y_2)+\Lambda_\psi\d(y_1,y_2)\\
	\geq & \frac{1}{2}s^2+(\Lambda_\psi-\Lip(\psi)-D)s=\frac{1}{2}s^2+s>0.
	\end{aligned}
	\end{equation}
	
	\medskip
	
	{\bf Step 3:} It holds that
	\begin{equation}\label{2.44}
		\lim_{t\rightarrow0} -t\log I_1=\psi^c.
	\end{equation}
Then we obtain \eqref{2.38} by combing \eqref{2.40}  and \eqref{2.44}.

	\paragraph{Lower bound:}
		\begin{equation}\label{2.45}
		\limi_{t\rightarrow0}	-t\log I_1\geq \psi^c.
	\end{equation}

		 Denote   $c_t(x,y)=-t \log p_\frac{t}{2}(x,y)$.   By Lemma \ref{lemma0.2},  we have $c_t(x,y)\rightarrow c(x,y)$ uniformly in $S\times Y_\delta$.    So for any $\epsilon>0$, there is $t_0>0$, such that for any $t<t_0$, 
		$c_t(x,y)\geq c(x,y)-\epsilon$, 
		and
		\begin{equation}\label{2.46}
			\begin{aligned}
				-t\log\int_{Y_\delta} e^{\frac{\psi_*(y)-c_t(x,y)}{t}}\d\mm(y)
				\geq &-t\log\int_{Y_\delta} e^{\frac{\psi_*(y)-c(x,y)+\epsilon}{t}}\,\d\mm(y)\\
				\overset{\eqref{3.61}}{\geq} &-t \log\int_{Y_\delta} e^{\frac{-\psi^c(x)+\epsilon}{t}}\,\d\mm(y)\\
				= & -t\left (\log\mm(Y_\delta)+\frac{-\psi^c(x)+\epsilon}{t}\right)\\
				=&\psi^c(x)-t\log\mm(Y_\delta)-\epsilon.
			\end{aligned}
		\end{equation}
		 Letting $t\rightarrow0$,  then letting $\epsilon\rightarrow0$, we get \eqref{2.45}.

		\paragraph{Upper bound:}
			\begin{equation}\label{2.47}
		\lims_{t\rightarrow0}	-t\log I_1\leq \psi^c.
	\end{equation}

		For  any $x\in S$,  by compactness of $Y$,  there exists  $T(x)\in Y$, such that
		$$\psi(T(x))+\psi^c(x)=c(x,T(x)).$$
		For any $\epsilon>0$, by continuity of $\psi_{*}$ and $c(x, y)$,  there is $r<\delta \wedge \epsilon$ such that
		$$\psi_{*}(y)\geq \psi(T(x))-\epsilon,\quad c(x,y)\leq c(x,T(x))+\epsilon,~~~\forall y\in B(T(x), r).$$
		Moreover, there exists $t_0>0$,  such that for any $t<t_0$, there holds
		$$c_t(x,y)\leq c(x,y)+\epsilon.$$
		
		Then for any $t<t_0$, we have
		\begin{equation}\label{2.48}
			\begin{aligned}
				-t\log\int_{Y_\delta} e^{\frac{\psi_*(y)-c_t(x,y)}{t}}\,\d\mm(y)
				\leq &-t\log\int_{B(T(x),r)} e^{\frac{\psi(T(x))-c(x,T(x))-3\epsilon}{t}}\,\d\mm(y)\\
				= &-t\log\int_{B(T(x),r)} e^{\frac{-\psi^c(x)-3\epsilon}{t}}\,\d\mm(y)\\
				=&-t\log\mm(B(T(x),r))+\psi^c(x)+3\epsilon.
			\end{aligned}
		\end{equation}
		Letting $t\rightarrow0$, then $\epsilon\rightarrow0$, we get \eqref{2.47}.

	\end{proof}

	\begin{lemma}\label{l0.21}
		Let $\varphi$ be a Kantorovich potentials from $\nu$ to $\rho$.   Then $\mu^t[\varphi_{*}]\weakto \nu$,   i.e.,  for any  $v\in\Lip(X,\d)$, it holds
		\begin{equation}
			\lim_{t\rightarrow0}\E_{\mu^t[\varphi_{*}]}(v)=\E_\nu(v).
		\end{equation}
	\end{lemma}
     \begin{proof}
     		We just need to show that 
     	\begin{equation}\label{0.38}
     		\lim_{t\rightarrow0}\E_{\mu_x^t[\varphi_{*}]}(v)=v(T(x)),\quad\text{for $\rho$-a.e. $x\in S$},
     	\end{equation}
     	where $T$ is the unique optimal transport map from $\rho$ to $\nu$.  The proof is very similar to that of Lemma \ref{l0.20}, so we only sketch it.
          \medskip
          
     	{\bf Step 1:}   Without loss of generality,  we assume that $v\geq 0$.  Similar to \eqref{2.40},  there is $\delta>0$ such that 
     	\begin{equation}\label{ext1}
     	\frac{\int_{X\setminus Y_\delta} v(y) e^{\frac{\varphi_*(y)}{t}}p_\frac{t}{2}(x,y)\,\d\mm(y)}{\int_X e^{\frac{\varphi_*(y)}{t}}p_\frac{t}{2}(x,y)\,\d\mm(y)}\lesssim  e^{-\frac{D^2}{t}} ~~\text{as}~~t\to 0.
     	\end{equation}
     	and
     	\begin{equation}\label{2.53}
     	\frac{\int_{X\setminus Y_\delta} e^{\frac{\varphi_*(y)}{t}}p_\frac{t}{2}(x,y)\,\d\mm(y)}{\int_X e^{\frac{\varphi_*(y)}{t}}p_\frac{t}{2}(x,y)\,\d\mm(y)}\lesssim  e^{-\frac{D^2}{t}} ~~\text{as}~~t\to 0.
     	\end{equation}

     	{\bf Step 2:} By uniqueness of optimal transport $T$ and \eqref{3.61}, for almost every $x\in S$,  there is a unique $T(x)\in Y$ so that $$\varphi^c(x)=c(x, T(x))-\varphi_*(T(x))=\inf_{y\in Y_\delta}\big\{c(x,y)-\varphi_{*}(y)\big\}.$$
    Then 	 for any $\epsilon>0$,   we can find
     	$0<r_2<\delta$ such that 
     	\begin{equation}\label{T}
     	|v(y)-v(T(x)| <\epsilon,~~~~~\forall y \in B\big(T(x),  r_2\big)
     	\end{equation}
     	and
     \begin{equation}\label{r1r2}
     	\varphi_*(y)-c(x,y)<-\varphi^c-4\delta_{r_2},~~~~~\forall y \in Y_\delta\setminus B\big(T(x),  r_2\big)
     	\end{equation}
     for some $	\delta_{r_2}>0$.
By Lemma \ref{lemma0.2},   there is  $t_0>0$ such that 
		$$|c_t(x,y)-c(x,y)|<\delta_{r_2},~~~\forall (x, y)\in S\times Y~\text{and}~ t<t_0.$$
		
By the same argument as \eqref{2.46},  for $t<t_0$ we have \begin{equation}\label{253}
  {e^{\frac{\varphi_*(y)}{t}}p_\frac{t}{2}(x,y)}\leq  {e^{\frac{-\varphi^c(x)-3\delta_{r_2}}{t}}},~~~\forall y\in Y_\delta\setminus B\big(T(x),  r_2\big),
     	\end{equation}
Furthermore,   similar to  \eqref{2.48},   we can find $ r_1<r_2$ so that
     	\begin{equation}\label{254}
  {e^{\frac{\varphi_*(y)}{t}}p_\frac{t}{2}(x,y)}> {e^{\frac{-\varphi^c(x)-2\delta_{r_2}}{t}}},~~~\forall y \in B\big(T(x),  r_1\big)~\text{and}~ t<t_0.
     	\end{equation}
     	
It follows from \eqref{253},  \eqref{254}  that
     \begin{equation}\label{255}
   	\frac{\int_{Y_\delta \setminus B(T(x),  r_2)} e^{\frac{\varphi_*(y)}{t}}p_\frac{t}{2}(x,y)\,\d\mm(y)}{\int_{B(T(x),  r_2)} e^{\frac{\varphi_*(y)}{t}}p_\frac{t}{2}(x,y)\,\d\mm(y)}\to 0 ~~\text{as}~~t\to 0.
     	\end{equation}
     	
Combining \eqref{ext1},  \eqref{2.53}, \eqref{T} and \eqref{255} we obtain
     	\begin{eqnarray*}
     	&&\lim_{t\rightarrow0}|\E_{\mu_x^t[\varphi_{*}]}(v)-v(T(x))|\\&=&\lim_{t\rightarrow0}  \left|	\frac{\int_{ X}  \big(v(y)-v(T(x))\big)e^{\frac{\varphi_*(y)}{t}}p_\frac{t}{2}(x,y)\,\d\mm(y)}{\int_{X} e^{\frac{\varphi_*(y)}{t}}p_\frac{t}{2}(x,y)\,\d\mm(y)}\right|\\ 
     	&=&\lim_{t\rightarrow0}  \left|	\frac{\int_{B(T(x),  r_2)}  \big(v(y)-v(T(x))\big)e^{\frac{\varphi_*(y)}{t}}p_\frac{t}{2}(x,y)\,\d\mm(y)}{\int_{B(T(x),  r_2)} e^{\frac{\varphi_*(y)}{t}}p_\frac{t}{2}(x,y)\,\d\mm(y)}\right|\\ &\leq & \epsilon.
     	\end{eqnarray*}
Letting $\epsilon \to 0$ we get \eqref{0.38}.

     \end{proof}
	
	\bigskip
	
	\begin{proof}[Proof of Theorem \ref{thm0.1}]
	Let  $\varphi$ and $\psi$ be  Kantorovich potentials from  $\mu$ to $\rho$ and  $\nu$ to $\rho$ respectively.
By the heat kernel lower bound \eqref{3.15} and  the stochastic completeness of the heat flow,  we can see that $|\Phi_t[\psi_{*}]|, |\Phi_t[\varphi_{*}]| $ are uniformly   bounded for $t\in(0,1]$. So by the dominated convergence theorem and Lemma \ref{l0.20} we obtain
		\begin{equation}\label{3.77}
			\begin{aligned}
				&\lim_{t\rightarrow0}\int_{S}\left|\Big(\Phi_t[\varphi_{*}](x)-\K_t[\varphi_{*}]\Big)-\Big(\Phi_t[\psi_{*}](x)-\K_t[\psi_{*}]\Big)\right|\,\d\rho(x)\\
				=&\int_{S}\left|\Big(\varphi^c(x)-\E_\rho(\varphi^c)\Big)-\Big(\psi^c(x)-\E_\rho(\psi^c)\Big)\right|\,\d\rho(x).
			\end{aligned}
		\end{equation}

		Note that  $(\varphi_{*}-\psi_{*})\restr{Y}=\varphi-\psi$,  by Lemma \ref{l0.21}, we have
		\begin{equation}\label{3.73}
		\lim_{t\rightarrow 0}\E_{\mu^t[\varphi_{*}]-\mu^t[\psi_{*}]}(\varphi_{*}-\psi_{*})=\E_{\mu-\nu}(\varphi-\psi).
		\end{equation}
		
		Combining Proposition
		\ref{0.30}, \eqref{3.77} and \eqref{3.73}, we obtain
		\begin{equation}
	\int_{S}\left|\Big(\varphi^c(x)-\E_\rho(\varphi^c)\Big)-\Big(\psi^c(x)-\E_\rho(\psi^c)\Big)\right|\,\d\rho(x)\leq  \sqrt{C\left|\E_{\mu-\nu}(\varphi-\psi)\right|}
		\end{equation}
		for some $C>0$.  Note that  $c(x,y)=\frac{1}{2}\d^2(x,y)$,  thus the Kantorovich potential $\varphi$ and $\psi$ are $\diam(S\cup Y)$-Lipschitz,  so $C$ depends only on $K,N,a_1,a_2,\diam(S\cup Y),S$.
		
		By assumption,  $\phi_\nu=\varphi^c, \phi_\mu=\psi^c$ satisfies  $\E_\rho(\phi_\mu)=\E_\rho(\phi_\nu)=0$.  Note also that $\Lip(\varphi-\psi)\leq 2\diam(S\cup Y)$, by Kantorovich duality for $W_1$, we finally obtain
		\begin{equation}
		\|\phi_\mu-\phi_\nu\|^2_{L^1(\rho)}\leq 2C \diam(S\cup Y) W_1(\mu,\nu).
		\end{equation}
	\end{proof}

	\section{Stability of optimal transport maps}\label{sec4}
	In this section, we prove Theorem \ref{thm0.2} concerning the quantitative stability of optimal transport maps on Alexandrov spaces. This is achieved by combining Theorem \ref{thm0.1} and the following  estimate.
	
	\begin{theorem}\label{thm2}
			Let $(X,  \d)$ be an Alexandrov space with no boundary. Then under the same assumptions of Theorem \ref{thm0.2},  there exists a constant $\bar{C}>0$, depending on $k$,$n$,$a_1$, $a_2$, $\diam(Y)$,  $S$, such that for any $\mu,\nu\in\mathcal{P}(Y)$, we have 
		\begin{equation}
			\int_S |\nabla\phi_\mu(x)-\nabla\phi_\nu(x)|^2\,\d\rho(x)\leq \bar{C} \left(\int_S|\phi_\mu(x)-\phi_\nu(x)|^2\,\d\rho(x)\right)^\frac{1}{3},
		\end{equation}
		where $\phi_\mu$ and $\phi_\nu$ are the Kantorovich potentials from $\rho$ to $\mu$ and $\rho$ to $\nu$ respectively.
	\end{theorem}
	
 We follow a strategy of \cite{arXiv:2504.05412} by lifting integrals from the base space to the unit tangent bundle,  evolve the relevant quantities via the geodesic flow developed in \cite{zbMATH07328107},  and then exploit the one-dimensional convexity of the potentials.  For the reader's convenience, we recall the necessary Alexandrov geometry theory and provide a self-contained proof.
	
	The Riemannian structure on the set of regular points $X_{\mathrm{reg}}$ (cf.\cite{zbMATH00606998}) endows  the tangent bundle $\mathrm{T}X_{\mathrm{reg}}$ an  Euclidean vector bundle structure.  On this Euclidean vector bundle, one can define (cf. \cite[Section3]{zbMATH07328107} ) a canonical Liouville measure $\mm_\mathrm L$: it is the unique Borel measure on $\mathrm{T}X_{\mathrm{reg}}$ such that for any Borel set $A\subset \mathrm{T}X_{\mathrm{reg}}$, we have
	\begin{equation*}
		\mm_\mathrm L(A)=\int_X \mathcal{H}^n\big(A\cap \mathrm{T}_xX\big)\,\d\mathcal{H}^n(x).
	\end{equation*}
	We extend $\mm_\mathrm L$ to a measure on the tangent bundle $\mathrm{T}X$ by setting $\mm_\mathrm L(\mathrm{T}X\setminus\mathrm{T}X_{\mathrm{reg}})=0$.   Moreover, by  \cite[Theorem 1.4]{zbMATH07694914}, if $(X,  \d)$ is an Alexandrov space with no boundary, then the geodesic flow preserves the Liouville measure. 
	
	Let $\mathrm{S}X=\{v\in \mathrm{T}X: |v|=1\}$ denote the unit tangent bundle (sphere bundle). There is a canonical Liouville measure $\mm_\mathrm S$ on $\mathrm{S}X$ (cf.\cite[Section 3.6]{zbMATH07328107}), also called Liouville measure, such that if the geodesic flow defined on ${\rm T}X$ preserves $\mm_\mathrm L$, then the geodesic flow defined on ${\mathrm S}X$ preserves $\mm_\mathrm S$ as well.   For any Borel set $A\subseteq {\mathrm S}X$, we have
	\begin{equation}\label{liouville}
		\mm_\mathrm S(A)=\int_{{\mathrm S}X}\chi_A(x,v)\,\d\mm_\mathrm S(x,v)=c_n\int_X\int_{\Sigma_x(X)} \chi_A(x,v)\,\d\sigma_x(v)\,\d\mathcal{H}^n(x),
	\end{equation}
	where $\Sigma_x(X)$ denotes the space of directions at $x$, $c_n=\mathcal{H}^{n-1}({\mathrm S}^{n-1})$ and $\sigma_x\in \mathcal{P}(\Sigma_x(X))$ is the canonical probability measure on the fiber.

	\paragraph{Lifting to the unit tangent bundle.}
	Let $\varphi_s: {\mathrm S}X\rightarrow{\mathrm S}X$ be the geodesic flow at time $s\in[0,1]$ on $\mathrm{S}X$, and write $\varphi_s(x,v)=(b_s(x,v),t_s(x,v))$, where $b_s(x,v)\in X$ and $t_s(x,v)\in {\Sigma_{b_s(x,v)}(X)}$. For $\mm_\mathrm S$-a.e. $(x,v)\in {\mathrm S}X$ such that the curve $[0,1]\ni s \mapsto b_s(x,v)$ is a locally shortest path,  denote by $I_S(x,v)$ the set of connected components of $\{s\in[0,1]: b_s(x,v)\in S\}$. Since $S$ is an open set, we have
	\begin{equation*}
		\{s\in [0,1]: b_s(x,v)\in S\}=\bigcup_{i\in I_S(x,v)}(\alpha_i(x,v),\beta_i(x,v)).
	\end{equation*}
	For  $x\in S \cap X_{\mathrm{reg}}$, by \cite[Theorem 10.8.4]{zbMATH01626771}, $\Sigma_x(X)$ is isometric to $\mathrm S^{n-1}$, then  there is a universal constant $d_n$ such that
	\begin{equation}\label{0.53}
		\begin{aligned}
			&\int_{\Sigma_x(X)} \langle\nabla\phi_\mu(x)-\nabla\phi_\nu(x),v\rangle^2\,\d\sigma_x(v)\\
			=&\int_{{\rS}^{n-1}}\langle\nabla\phi_\mu(x)-\nabla\phi_\nu(x),v\rangle^2\,\d\sigma_x(v)\\
			=&d_n |\nabla\phi_\mu(x)-\nabla\phi_\nu(x)|^2.
		\end{aligned}
	\end{equation}

	By \cite[Section 3.6]{zbMATH07328107},   $\mm_\rS$ is preserved by $\varphi_s$,  so
	\begin{equation}\label{0.54}
		\begin{aligned}
			&\int_S |\nabla\phi_\mu(x)-\nabla\phi_\nu(x)|^2\,\d\rho(x)
			\leq  a_2\int_S |\nabla\phi_\mu(x)-\nabla\phi_\nu(x)|^2\,\d\mathcal{H}^n(x)\\
			\mathop{=}^{\eqref{0.53}}&\frac{a_2}{d_n}\int_{S}\int_{\Sigma_x(X)}\langle\nabla\phi_\mu(x)-\nabla\phi_\nu(x),v\rangle^2\,\d\sigma_x(v)\,\d\mathcal{H}^n(x)\\
			\mathop{=}^{\eqref{liouville}}&\frac{a_2}{c_nd_n}\int_{{\rS}X}\langle\nabla\phi_\mu(x)-\nabla\phi_\nu(x),v\rangle^2\chi_S(x)\,\d\mm_\rS(x,v)\\
			=&\frac{a_2}{c_nd_n}\int_{{\rS}X}\int_0^1\langle\nabla\phi_\mu(b_s(x,v))-\nabla\phi_\nu(b_s(x,v)),t_s(x,v)\rangle^2\chi_S(b_s(x,v))\,\d s \,\d\mm_\rS(x,v)\\
			=&\frac{a_2}{c_nd_n}\int_{{\rS}X}\sum_{i\in  I_S(x,v)}\int_{\alpha_i(x,v)}^{\beta_i(x,v)}\langle\nabla\phi_\mu(b_s(x,v))-\nabla\phi_\nu(b_s(x,v)),t_s(x,v)\rangle^2\,\d s \,\d\mm_\rS(x,v).\\	
		\end{aligned}
	\end{equation}
	
	\paragraph{One-dimensional convexity estimate.}
	
	For $(x,v)\in {\rS}X$ and $s\in[0,1]$,   we denote $u_\mu^{(x,v)}(s)=\phi_\mu(b_s(x,v))$.   Then
	\begin{equation*}
		\frac{\d}{\d s}u_\mu^{(x,v)}(s)=\langle\nabla\phi_\mu(b_s(x,v)), t_s(x,v)\rangle, \quad \text{for a.e. $s\in [0,1]$},
	\end{equation*}
	and an analogous formula holds for $u_\nu^{(x,v)}(s)=\phi_\nu(b_s(x,v))$.   
	
		Denote $S_1=\{x\in X: \d(x, S)\leq 1\}$.  Note that in the last integral of \eqref{0.54}, only the elements $(x,v)\in {\rS}X$ for which $x\in S_1$ have a non-vanishing contribution.  For $\mm_\rS$-a.e. $(x,v)\in {\rS}X$ with $x\in S_1$, the curve $s\mapsto b_s(x,v)$ is locally minimizing on $[0,1]$. Thus, applying Ohta's semiconcavity estimate  \cite[Lemma 3.2]{zbMATH05549756}  on sufficiently small subsegments yields local concavity of $u^{(x,v)}_\mu-\zeta |s|^2$ with a uniform constant $\zeta$. The local-to-global principle for concave functions then leads to the following lemma.

	\begin{lemma}[\cite{zbMATH05549756}, Lemma 3.2]\label{l0.17}
		There exists $\zeta$, depending on $k$ and $\diam(S\cup Y)$, such that for $\mm_\rS$-a.e. $(x,v)\in {\rS}{S_1}$, the functions $u_\mu^{(x,v)}-\zeta|s|^2$, $u_\nu^{(x,v)}-\zeta|s|^2$ are concave on $s\in[0,1]$. Moreover, the modulus of their derivatives (which exist a.e. on $[0,1]$) is bounded above by $\diam(S\cup Y)+2\zeta $.
	\end{lemma}

	Applying Lemma \ref{l0.18} below to the functions $\zeta|s|^2-u_\mu^{(x,v)}$ and $\zeta|s|^2-u_\nu^{(x,v)}$ on each compact segment $[\alpha_i(x,v),\beta_i(x,v)]$ and using H\"{o}lder inequality, we obtain
	\begin{equation}\label{0.55}
		\begin{aligned}
			&\int_S |\nabla\phi_\mu(x)-\nabla\phi_\nu(x)|^2\,\d\rho(x)\\
			\mathop{\leq}^{\eqref{0.54}}& \frac{a_2}{c_nd_n}\int_{{\rS}X}\sum_{i\in  I_S(x,v)}\int_{\alpha_i(x,v)}^{\beta_i(x,v)}\langle\nabla\phi_\mu(b_s(x,v))-\nabla\phi_\nu(b_s(x,v)),t_s(x,v)\rangle^2\,\d s \,\d\mm_\rS(x,v)\\
			\leq&C_1 \int_{{\rS}X}\sum_{i\in  I_S(x,v)}\left(\int_{\alpha_i(x,v)}^{\beta_i(x,v)}|\phi_\mu(b_s(x,v))-\phi_\nu(b_s(x,v))|^2\,\d s\right)^\frac{1}{3}\,\d\mm_\rS(x,v)\\
			\leq&C_1 \int_{{\rS}X}\left(\# I_S\right)^\frac{2}{3}\left(\int_0^1|\phi_\mu(b_s)-\phi_\nu(b_s)|^2\chi_S(b_s)\,\d s\right)^\frac{1}{3}\d\mm_\rS\\
			\leq&C_1 \left(\int_{{\rS}X}\# I_S\,\d\mm_\rS\right)^\frac{2}{3}\left(\int_{{\rS}X}\int_0^1|\phi_\mu(b_s)-\phi_\nu(b_s)|^2\chi_S(b_s)\,\d s \,\d\mm_\rS\right)^\frac{1}{3}\\
			\leq&C_1 \left(\frac{c_n}{a_1}\right)^\frac{1}{3}\left(\int_{{\rS}X}\# I_S\,\d\mm_\rS\right)^\frac{2}{3}\left(\int_S|\phi_\mu-\phi_\nu|^2\,\d\rho\right)^\frac{1}{3},
		\end{aligned}
	\end{equation}
	where $C_1$ depends on $k,n,a_2, \diam(S\cup Y)$, and the last inequality follows from the invariance of $\mm_\rS$ under the geodesic flow.
	
	\begin{lemma}[\cite{zbMATH07794624}, Lemma 5.1]\label{l0.18}
		Let $I\subseteq \R$ be a compact segment and let $u,v: I\rightarrow\R$ be two convex functions such that $|u'|$ and $|v'|$ (defined a.e. on $I$) are uniformly bounded on $I$. Then
		\begin{equation*}
			\|u'-v'\|_{L^2(I)}^2\leq 8 (\|u'\|_{L^\infty(I)}+\|v'\|_{L^\infty(I)})^\frac{4}{3}\|u-v\|_{L^2(I)}^\frac{2}{3}.
		\end{equation*}
	\end{lemma}
	\paragraph{One-dimensional BV estimate.}

	We now aim to control $\int_{{\rS}X}\# I_S\,\d\mm_\rS$,  the average number of times a geodesic crosses the boundary of $S$. The argument of    \cite[Proposition 4.3]{arXiv:2504.05412} relies  on the existence of $T >0$, so that for $(x,v)\in {\rS}X$, the geodesic $s \mapsto b_s(x,v)$ is minimizing on $[0,T]$, and hence in particular does not self-intersect.  In an Alexandrov space, one cannot in general expect such a uniform injectivity radius bound. We  use a localization method to  overcome this difficulty.

	In the following proposition,  we denote by $|\D f|([0,1])$ the total variation of a function $f$ on $[0,1]$ and denote by ${\rm{Per}}(S)$ the perimeter of a set $S\subset X$.   We refer to \cite{zbMATH05152934, zbMATH06320681} for $\rm{BV}$ functions and sets of finite perimeter in metric measure spaces.

	\begin{proposition}\label{prop3.5}
		Assume ${\rm{Per}}(S)<+\infty$, then 
		\begin{equation*}
			\int_{{\rS}X}\# I_S(x,v)\,\d\mm_\rS(x,v)\leq c_n\big(\mathcal{H}^n(S_1)+{\rm{Per}}(S)\big).
		\end{equation*}
	\end{proposition}
	
	\begin{proof}
 Let  $(x,v)\in {\rS}X$ be such that  $s \mapsto b_s(x,v)$ is well-defined on $[0,1]$, and let  $u\in \Lip(X,\d)$. 
		Note that $s \mapsto b_s(x,v)$ has unit speed, then 
		\begin{equation*}
			\left|\frac{\d}{\d s}u(b_s(x,v))\right|\leq \left|\nabla u\right|(b_s(x,v)),~~~\text{a.e. $s\in [0,1]$}, 
		\end{equation*} and
		\begin{equation}\label{3.7}
		|\D u (b_s(x,v))|([0,1])=\int_0^1 	\left|\frac{\d}{\d s}u(b_s(x,v))\right|\,\d s\leq  \int_0^1 |\nabla u|(b_s(x,v))\,\d s.
			\end{equation}
		
Let $(u_k)_{k\in \N}\subseteq \Lip(X,\d)$ be such that $u_k\rightarrow \chi_S$ in $L^1(\mathcal H^n)$ and
\begin{equation}\label{Per(S)}
			 \int_X |\nabla u_k|\,\d \mathcal{H}^n \rightarrow  {\rm {Per}}(S).
			\end{equation}Then
		\begin{equation}
			\begin{aligned}
			 &\int_{{\rS}X} \int_0^1\left|u_k(b_s(x,v))-\chi_S(b_s(x,v))\right|\,\d s \,\d\mm_\rS(x,v)\\
			 =& \int_0^1\int_{{\rS}X}\left|u_k(b_s(x,v))-\chi_S(b_s(x,v))\right|\,\d\mm_\rS(x,v)\,\d s\\
			{=}&c_n \int_X |u_k(x)-\chi_S(x)|\,\d\mathcal{H}^n(x)\rightarrow 0.
			\end{aligned}
		\end{equation}
	So up to taking a subsequence,   it holds that for $\mm_\rS$-a.e. $(x,v)\in {\rS}X$,
	\begin{equation}
		u_k(b_s(x,v))\rightarrow \chi_S(b_s(x,v)) \quad\text{in $L^1([0,1])$ as $k\to \infty$}.
	\end{equation}

By  lower semicontinuity of the total variation, we have
		\begin{equation}\label{3.10}
		\begin{aligned}
		&\int_{{\rS}X}|\D\chi_S(b_s(x,v))|([0,1])\,\d\mm_\rS(x,v)\\
		\leq &\int_{{\rS}X}\limi_{k\rightarrow\infty}|\D u_k(b_s(x,v))|([0,1])\,\d\mm_\rS(x,v)\\
	\mathop{\leq}^{\text{Fatou}} &\limi_{k\rightarrow\infty}\int_{{\rS}X}|\D u_k(b_s(x,v))|([0,1])\,\d\mm_\rS(x,v)\\
		\mathop{\leq}^* &\limi_{k\rightarrow\infty}c_n \int_X |\nabla u_k|\,\d \mathcal{H}^n \mathop{=}^{\eqref{Per(S)}} c_n{\rm{Per}}(S),
		\end{aligned}
		\end{equation}
			where $(*)$ follows from \eqref{3.7} and the invariance of $\mm_\rS$ under the geodesic flow.
			\medskip
			
	Finally,  notice that
		\begin{equation}\label{36}
		\# I_S(x,v)\leq 1+\frac{1}{2}|\D\chi_S(b_s(x,v))|([0,1])
		\end{equation} and that only the elements $(x,v)\in {\rS}X$ for which $x\in S_1$ have a non-vanishing contribution.  Combining with \eqref{3.10},  we obtain 
	\begin{equation}
		\begin{aligned}
			\int_{{\rS}X}\# I_S(x,v)\,\d\mm_\rS(x,v)
			\leq & c_n\mathcal{H}^n(S_1)+\frac{1}{2}c_n {\rm{Per}}(S).
		\end{aligned}
	\end{equation}
	\end{proof}
	
		\paragraph{Proof of the theorems.}
	\begin{proof}[Proof of Theorem \ref{thm2}]
		By \eqref{0.55} and Proposition \ref{prop3.5}, we obtain
		\begin{equation}
			\begin{aligned}
				&\int_S |\nabla\phi_\mu-\nabla\phi_\nu|^2\,\d\rho\\
				\leq&C_1 \left(\frac{c_n}{a_1}\right)^\frac{1}{3}\left(\int_{{\rS}X}\# I_S\,\d\mm_\rS\right)^\frac{2}{3}\left(\int_S|\phi_\mu-\phi_\nu|^2\,\d\rho\right)^\frac{1}{3}\\
				\leq&	 \bar{C}\left(\int_S|\phi_\mu-\phi_\nu|^2\,\d\rho\right)^\frac{1}{3},
			\end{aligned}
		\end{equation}
		where $\bar{C}$ depends on $k,n,a_1, a_2, \diam(S\cup Y), {\rm{Per}}(S), S$.
	\end{proof}

		\begin{proof}[Proof of Theorem \ref{thm0.2}]
		By Perelman's doubling theorem (or Petrunin's gluing theorem \cite{zbMATH01104278}),  we may assume that $(X,  \d)$ has no boundary.  
		
		Let $\phi_\mu$ and $\phi_\nu$ be the Kantorovich potentials from $\rho$ to $\mu$ and $\rho$ to $\nu$ respectively.   From \cite{zbMATH05349267,zbMATH06032507} we know that  $\nabla \phi_\mu(x),  \nabla \phi_\nu(x) \in  {\rm{T}}_xX$  and
		$T_\mu(x)= \exp_x(-\nabla \phi_\mu(x))$,  $T_\nu(x)= \exp_x(-\nabla \phi_\nu(x))$  are well-defined for almost every $x\in S$.   By triangle comparison condition (cf. \cite{zbMATH05342782}),  it holds
		\begin{equation}\label{exp}
			\d \big(T_\mu(x),   T_\nu(x) \big) \leq c|\nabla \phi_\mu(x)-\nabla \phi_\nu(x)|,
		\end{equation}
		for some constant $c>0$ which depends on $k, \diam(S\cup Y)$.  
		\medskip

Since $\phi_\mu,\phi_\nu$ are $\diam(S\cup Y)$-Lipschitz and  $\E_\rho(\phi_\mu)=\E_\rho(\phi_\nu)=0$, we have  
	\begin{equation*}
		\|\phi_\mu-\phi_\nu\|_{L^\infty(\rho)}\leq \|\phi_\mu\|_{L^\infty(\rho)}+ \|\phi_\nu\|_{L^\infty(\rho)}\leq {\rm{osc}}(\phi_\mu)+{\rm{osc}}(\phi_\nu)\leq 2\big(\diam(S\cup Y)\big)^2.
	\end{equation*}
Recall that a finite dimensional Alexandrov space is also  $\rm{RCD}$ (cf.  \cite{zbMATH06032507,zbMATH05962558}).    By Theorem \ref{thm0.1},   we  get
	\begin{equation}\label{sta}
		\|\phi_\mu-\phi_\nu\|^2_{L^2(\rho)} \leq \|\phi_\mu-\phi_\nu\|_{L^\infty(\rho)}\|\phi_\mu-\phi_\nu\|_{L^1(\rho)} \leq C_3 W_1^\frac{1}{2}(\mu,\nu).
	\end{equation}

		 Combining  Theorem \ref{thm2},   \eqref{exp} and  \eqref{sta}, we obtain
		\begin{equation}
			\begin{aligned}
				\int_S \d^2(T_\mu(x), T_\nu(x))\,\d \rho(x)	\leq C W_1^\frac{1}{6}(\mu,\nu),
			\end{aligned}
		\end{equation}
		where $C$ depends on $k,n,a_1, a_2, \diam(S\cup Y), {\rm{Per}}(S), S$.  This complete the proof.
	\end{proof}

	\begin{appendices}
	\section{Poincar\'e inequality: local to global}\label{A}
		\begin{lemma}\label{p3.10}
	Let $x_0\in S$, $0<r_0\leq1$ and  $B=B(x_0,r_0)\subseteq S$,  $\rho_B=\frac{\rho\restr{B}}{\rho(B)}$.  Then the following strong local $(1,1)$-Poincar\'e inequality holds for $f\in \Lip(B,  \d)$:
			\begin{equation}
			\int_{B}|f(x)-\E_{\rho_B}(f)|\,\d\rho_B(x)\leq C_3 r_0\int_B |\nabla f|\,\d\rho_B,
		\end{equation}
		where $C_3$ depends on $K,N,a_1,a_2$.
	\end{lemma}
	
	\begin{proof}
		By \cite[Chapter 9]{zbMATH01474795} and \cite[Remark 3.3]{zbMATH06043352},  $\mm$ satisfies the strong local  $(1,1)$-Poincar\'e inequality:
		\begin{equation}
			\int_{B} |f-\E_{\mm_B}(f)|\,\d\mm\leq c_1 r_0 \int_B |\nabla f|\,\d\mm,
		\end{equation}
		where $c_1$ depends on $K,N$, $\mm_B=\frac{\mm\restr{B}}{\mm(B)}$. Since $a_1\mm\restr{S} \leq \rho\leq a_2\mm\restr{S}$, we have 
		\begin{equation}
			\begin{aligned}
				\int_{B} |f-\E_{\rho_B}(f)|\,\d\rho\leq &\int_{B} |f-\E_{\mm_B}(f)|\,\d\rho+\int_{B} |\E_{\mm_B}(f)-\E_{\rho_B}(f)|\,\d\rho\\
				\leq & 2a_2\int_{B} |f-\E_{\mm_B}(f)|\,\d\mm\\
				\leq &2a_2 c_1 r_0 \int_B |\nabla f|\,\d\mm
				\leq  2\frac{a_2}{a_1}c_1 r_0 \int_B |\nabla f|\,\d\rho,\\
			\end{aligned}
		\end{equation}
	which is the thesis.
	\end{proof}
	
	Next we prove an $L^1$-variant of the gluing Lemma in \cite[Lemma 3.3]{arXiv:2411.04908},   following the same Boman-chain decomposition,  utilizing the doubling property directly instead of the maximal function estimates. 
	
		\begin{definition}[Boman chain condition]\label{Boman chain condition}
		We say that a probability measure $\rho$ on an open set $S$ of a metric space satisfies  \emph{Boman chain condition} with parameters \(E, F,G>1\) if there is a covering \(\mathcal{F}\) of $S$ by open balls  $B \in \mathcal{F}$ such that:
	
		\begin{enumerate}
				\item For any \(x \in S\),
				\begin{equation*}
						\sum_{B \in \mathcal{F}} \chi_{2B}(x) \leq E \chi_{S}(x).
					\end{equation*}
				\item For some fixed ball \(B_0\) in \(\mathcal{F}\), called the \emph{central ball}, and for every \(B \in \mathcal{F}\), there exists a chain \(B_0, B_1, \ldots, B_N = B\) of distinct balls from \(\mathcal{F}\) such that
				\begin{equation*}
					B \subset F B_j,~~~\forall j \in \{0, \ldots, N-1\}.
					\end{equation*}
			\item Consecutive balls of the above chain overlap quantitatively:
				\begin{equation*}
						\rho(B_j \cap B_{j+1}) \geq G^{-1} \max(\rho(B_j), \rho(B_{j+1})),~~~\forall j \in \{0, \ldots, N-1\}.
					\end{equation*}
			\end{enumerate}
	\end{definition}
	
	\begin{lemma}\label{3.1}
	For any $f\in L^1(\rho)$,   it holds that
	\begin{equation}\label{A.5}
		\int_S |f(x)-\E_{\rho}(f)|\,\d\rho(x)\leq C_4 \sum_{B\in\mathcal{F}}\rho(B)\int_B |f(x)-\E_{\rho_B}(f)|\,\d\rho_B(x),
	\end{equation}
	where $C_4$ depends on $K,N,\diam(S),S$.
	\end{lemma}
	
    \begin{proof} 
    	Since $S$ is a John domain, and $\rho$ is a doubling measure on $S$, by  \cite[Proposition 3.7]{arXiv:2504.05412}, $\rho$ satisfies the Boman chain condition. Hence there exists a covering $\mathcal{F}$ of $S$ satisfying Definition \ref{Boman chain condition}.
    	
    	For the central ball $B_0$, note that
    	\begin{equation}\label{A.2}
    		\begin{aligned}
    			\int_S |f(x)-\E_{\rho}(f)|\,\d\rho(x)
    			\leq & \int_S |f(x)-\E_{\rho_{B_0}}(f)|\,\d\rho(x)+\int_S |\E_{\rho_{B_0}}(f)-\E_{\rho}(f)|\,\d\rho(x)\\
    			\leq & 2\int_S |f(x)-\E_{\rho_{B_0}}(f)|\,\d\rho(x).
    		\end{aligned}
    	\end{equation}
    	
   For  $B\in\mathcal{F}$,  	denote $a_B=\int_B |f-\E_{\rho_{B}}(f)|\,\d\rho=\rho(B)\int_B |f-\E_{\rho_{B}}(f)|\,\d\rho_B$. Since $\mathcal{F}$ is a covering of $S$, we have
    	\begin{equation}\label{A.3}
    		\begin{aligned}
    			&\int_S |f(x)-\E_{\rho_{B_0}}(f)|\,\d\rho(x)
    			\leq  \sum_{B \in \mathcal{F}} \int_B |f(x)-\E_{\rho_{B_0}}(f)|\,\d\rho(x)\\
    			\leq & \sum_{B \in \mathcal{F}} \left(\int_B |f(x)-\E_{\rho_{B}}(f)|\,\d\rho(x)+\int_B |\E_{\rho_{B}}(f)-\E_{\rho_{B_0}}(f)|\,\d\rho(x)\right)\\
    			\leq &  \sum_{B \in \mathcal{F}} \left(a_B+\rho(B) |\E_{\rho_{B}}(f)-\E_{\rho_{B_0}}(f)|\right).
    		\end{aligned}
    	\end{equation}
    	
    	For any $B\in\mathcal{F}$, by Boman chain condition, there exists a chain $B_0, B_1,\dots, B_N=B$ of distinct balls from \(\mathcal{F}\), such that for any \(j \in \{0, \ldots, N-1\}\),
    	\begin{equation}
    		\begin{aligned}
    		|\E_{\rho_{B_j}}(f)-\E_{\rho_{B_{j+1}}}(f)|= &\left|\frac{1}{\rho(B_j\cap B_{j+1})}\int_{B_j\cap B_{j+1}}\left(\E_{\rho_{B_j}}(f)-\E_{\rho_{B_{j+1}}}(f)\right)\,\d\rho\right|\\
    		\leq &\frac{1}{\rho(B_j\cap B_{j+1})}\int_{B_j\cap B_{j+1}}\left|\E_{\rho_{B_j}}(f)-\E_{\rho_{B_{j+1}}}(f)\right|\,\d\rho\\
    		\leq &\frac{1}{\rho(B_j\cap B_{j+1})}\left(\int_{B_j\cap B_{j+1}}\left|f-\E_{\rho_{B_j}}(f)\right|+\left|f-\E_{\rho_{B_{j+1}}}(f)\right|\,\d\rho\right)\\
    		\leq & \frac{a_{B_j}+a_{B_{j+1}}}{\rho(B_j\cap B_{j+1})}
    		\overset{*}{\leq}  G \left(\frac{a_{B_j}}{\rho(B_j)}+\frac{a_{B_{j+1}}}{\rho(B_{j+1})}\right),
    		\end{aligned}
    	\end{equation}
    	where $(*)$ follows from the quantitative chain overlap of Boman chain condition $3$.
    	 
    	 Thus, we have
    	 \begin{equation}
    	 |\E_{\rho_{B}}(f)-\E_{\rho_{B_0}}(f)|\leq \sum_{j=0}^{N-1}|\E_{\rho_{B_j}}(f)-\E_{\rho_{B_{j+1}}}(f)|\leq 2G \sum_{j=0}^{N}\frac{a_{B_j}}{\rho(B_j)}\overset{*}{\leq} 2G \sum_{B\subset F \bar{B}}\frac{a_{\bar{B}}}{\rho(\bar{B})},
    	 \end{equation}
    	 where $ \sum_{B\subset F \bar{B}}$ means that the sum runs over all  $\bar{B}\in\mathcal{F}$ satisfying $B\subset F \bar{B}$, and $(*)$ follows from the Boman chain condition $2$.
    	 Then by Fubini--Tonelli theorem,  
    	 \begin{equation}\label{A.6}
    	 	\begin{aligned}
    	 	 	  \sum_{B \in \mathcal{F}} \rho(B) |\E_{\rho_{B}}(f)-\E_{\rho_{B_0}}(f)|
    	 	 	  \leq & 2G\sum_{B \in \mathcal{F}} \rho(B) \sum_{B\subset F \bar{B}}\frac{a_{\bar{B}}}{\rho(\bar{B})}
    	 	 	  \leq 2G\sum_{\bar{B} \in \mathcal{F}} \frac{a_{\bar{B}}}{\rho(\bar{B})}\sum_{B\subset F \bar{B}}\rho(B).
    	 	\end{aligned}
    	 \end{equation}
    	 
By Boman chain condition $1$ and  the doubling property of $\rho$, we have  
    	 \begin{equation}\label{aa}
       \sum_{B\subset F \bar{B}}\rho(B)\leq E\rho(F\bar{B})\leq E\beta^2F^\frac{\log \beta}{\log 2}\rho(\bar{B}),
    	 \end{equation}
where  $\beta=\beta(K,N,\diam(S))$ is the doubling constant.
    	 
    	 Combining \eqref{A.2}, \eqref{A.3},  \eqref{A.6} and \eqref{aa}, we obtain
    	 \begin{equation}\label{A_A}
    	 	\begin{aligned}
    	 		 \int_S |f(x)-\E_{\rho}(f)|\,\d\rho(x)\leq & 2 \sum_{B \in \mathcal{F}} \left(a_B+\rho(B) |\E_{\rho_{B}}(f)-\E_{\rho_{B_0}}(f)|\right)\\
    	 		 \leq &2(1+2\beta^2 EF^\frac{\log \beta}{\log 2}G) \sum_{B \in \mathcal{F}} \rho(B)\int_B |f-\E_{\rho_{B}}(f)|\,\d\rho_B,
    	 	\end{aligned}
    	 \end{equation}
  which is the thesis.
    \end{proof}

	\end{appendices}

	\addcontentsline{toc}{section}{References}
\def\cprime{$'$}

	\bigskip

\noindent {\bf Declaration.} The authors declare no conflict of interest and that the manuscript has no associated data.
	
	\medskip
	
	\noindent {\bf Fundings.} This work is supported by the Ministry of Science \& Technology of China (2021YFA1000900, 2021YFA1002200),  National Natural Science Foundation of China  (12201596)   and Shandong Provincial Natural Science Foundation (ZR2025QB05).

\end{document}